\documentclass[11pt,twoside,a4paper]{article}
\usepackage{sparticles}
\usepackage{feynmf}	
\usepackage{amsmath}
\usepackage{amssymb}
\usepackage[latin1]{inputenc}
\usepackage[T1]{fontenc}   
\usepackage[english]{babel}
\newcommand{\tun}{\begin{picture}(5,0)(-2,-1)
\put(0,0){\circle*{2}}
\end{picture}}

\newcommand{\tdeux}{\begin{picture}(7,7)(0,-1)
\put(3,0){\circle*{2}}
\put(3,0){\line(0,1){5}}
\put(3,5){\circle*{2}}
\end{picture}}

\newcommand{\tdun}[1]{\begin{picture}(10,5)(-2,-1)
\put(0,0){\circle*{2}}
\put(3,-2){\tiny #1}
\end{picture}}

\newcommand{\tddeux}[2]{\begin{picture}(12,5)(0,-1)
\put(3,0){\circle*{2}}
\put(3,0){\line(0,1){5}}
\put(3,5){\circle*{2}}
\put(6,-2){\tiny #1}
\put(6,3){\tiny #2}
\end{picture}}

\newcommand{\tdtroisun}[3]{\begin{picture}(20,12)(-5,-1)
\put(3,0){\circle*{2}}
\put(-0.65,0){$\vee$}
\put(6,7){\circle*{2}}
\put(0,7){\circle*{2}}
\put(5,-2){\tiny #1}
\put(9,5){\tiny #2}
\put(-5,5){\tiny #3}
\end{picture}}
\newcommand{\tdtroisdeux}[3]{\begin{picture}(12,12)(-2,-1)
\put(0,0){\circle*{2}}
\put(0,0){\line(0,1){5}}
\put(0,5){\circle*{2}}
\put(0,5){\line(0,1){5}}
\put(0,10){\circle*{2}}
\put(3,-2){\tiny #1}
\put(3,3){\tiny #2}
\put(3,9){\tiny #3}
\end{picture}}

\newcommand{\tdquatreun}[4]{\begin{picture}(20,12)(-5,-1)
\put(3,0){\circle*{2}}
\put(-0.6,0){$\vee$}
\put(6,7){\circle*{2}}
\put(0,7){\circle*{2}}
\put(3,7){\circle*{2}}
\put(3,0){\line(0,1){7}}
\put(5,-2){\tiny #1}
\put(8.5,5){\tiny #2}
\put(1,10){\tiny #3}
\put(-5,5){\tiny #4}
\end{picture}}
\newcommand{\tdquatredeux}[4]{\begin{picture}(20,20)(-5,-1)
\put(3,0){\circle*{2}}
\put(-.65,0){$\vee$}
\put(6,7){\circle*{2}}
\put(0,7){\circle*{2}}
\put(0,14){\circle*{2}}
\put(0,7){\line(0,1){7}}
\put(5,-2){\tiny #1}
\put(9,5){\tiny #2}
\put(-5,5){\tiny #3}
\put(-5,12){\tiny #4}
\end{picture}}
\newcommand{\tdquatretrois}[4]{\begin{picture}(20,20)(-5,-1)
\put(3,0){\circle*{2}}
\put(-.65,0){$\vee$}
\put(6,7){\circle*{2}}
\put(0,7){\circle*{2}}
\put(6,14){\circle*{2}}
\put(6,7){\line(0,1){7}}
\put(5,-2){\tiny #1}
\put(9,5){\tiny #2}
\put(-5,5){\tiny #4}
\put(9,12){\tiny #3}
\end{picture}}
\newcommand{\tdquatrequatre}[4]{\begin{picture}(20,14)(-5,-1)
\put(3,5){\circle*{2}}
\put(-.65,5){$\vee$}
\put(6,12){\circle*{2}}
\put(0,12){\circle*{2}}
\put(3,0){\circle*{2}}
\put(3,0){\line(0,1){5}}
\put(6,-3){\tiny #1}
\put(6,4){\tiny #2}
\put(9,12){\tiny #3}
\put(-5,12){\tiny #4}
\end{picture}}
\newcommand{\tdquatrecinq}[4]{\begin{picture}(12,19)(-2,-1)
\put(0,0){\circle*{2}}
\put(0,0){\line(0,1){5}}
\put(0,5){\circle*{2}}
\put(0,5){\line(0,1){5}}
\put(0,10){\circle*{2}}
\put(0,10){\line(0,1){5}}
\put(0,15){\circle*{2}}
\put(3,-2){\tiny #1}
\put(3,3){\tiny #2}
\put(3,9){\tiny #3}
\put(3,14){\tiny #4}
\end{picture}}

\newcommand{\photonboson}{\begin{fmffile}{photonboson}\parbox{12mm}{\begin{fmfgraph}(20,20)
\fmfleft{i1}
\fmfright{o1,o2}
\fmf{photon}{i1,v}
\fmf{fermion}{o1,v,o2}
\fmfblob{.3w}{v}
\end{fmfgraph}}\end{fmffile}}

\newcommand{\photon}{\begin{fmffile}{photon}\parbox{12mm}{\begin{fmfgraph}(20,20)
\fmfleft{i}
\fmfright{o}
\fmf{photon}{i,v}
\fmf{photon}{v,o}
\fmfblob{.3w}{v}
\end{fmfgraph}}\end{fmffile}}

\newcommand{\boson}{\begin{fmffile}{boson}\parbox{12mm}{\begin{fmfgraph}(20,20)
\fmfleft{i}
\fmfright{o}
\fmf{fermion}{i,v}
\fmf{fermion}{v,o}
\fmfblob{.3w}{v}
\end{fmfgraph}}\end{fmffile}}

\newcommand{\graphedeux}{\begin{fmffile}{diag4}\parbox{12mm}{\begin{fmfgraph}(30,30)
\fmfleft{i}\fmfright{o}\fmf{photon}{i,v1}\fmf{photon}{v2,o}\fmf{fermion,left,tension=.2}{v1,v2,v1}
\end{fmfgraph}}\end{fmffile}}

\newcommand{\graphetrois}{\begin{fmffile}{diag6}\parbox{16mm}{\begin{fmfgraph}(40,30)
\fmfleft{i}\fmfright{o}\fmf{fermion}{i,v1,v2,o}\fmf{photon,left,tension=0. }{v1,v2}
\end{fmfgraph}}\end{fmffile}}

\input{xy}
\xyoption{all}

\newcommand{\B}{\mathcal{B}}
\newcommand{\U}{\mathcal{U}}
\newcommand{\h}{\mathcal{H}}
\newcommand{\J}{\mathcal{J}}
\newcommand{\g}{\mathfrak{g}}
\newcommand{\hs}{\mathcal{H}_{(S)}}
\newcommand{\gs}{G_{(S)}}
\newcommand{\fleche}[1]{\stackrel{#1}{\longrightarrow}}

\title{General Dyson-Schwinger equations and systems}
\date{}
\author{Loïc Foissy\footnote{e-mail: loic.foissy@univ-reims.fr; webpage: http://loic.foissy.free.fr/pageperso/accueil.html}
\\
{\small{\it Laboratoire de Mathématiques, Université de Reims}}\\
\small{{\it Moulin de la Housse - BP 1039 - 51687 REIMS Cedex 2, France}}\\
}

\newtheorem{defi}{\indent Definition}
\newtheorem{lemma}[defi]{\indent Lemma}
\newtheorem{cor}[defi]{\indent Corollary}
\newtheorem{theo}[defi]{\indent Theorem}
\newtheorem{prop}[defi]{\indent Proposition}

\newenvironment{proof}{{\bf Proof.}}{\hfill $\Box$}

\begin{document}
\maketitle

ABSTRACT. We classify combinatorial Dyson-Schwinger equations giving a Hopf subalgebra of the Hopf algebra of Feynman graphs of the considered Quantum Field Theory.
We first treat single equations with an arbitrary number (eventually infinite) of insertion operators.
we distinguish two cases; in the first one, the Hopf subalgebra generated by the solution is isomorphic to the Faà di Bruno Hopf algebra
or to the Hopf algebra of symmetric functions; in the second case, we obtain the dual of the enveloping algebra of a particular associative algebra (seen as a Lie algebra).
We also treat systems with an arbitrary finite number of equations, with an arbitrary number of insertion operators, 
with at least one of degree $1$ in each equation. \\



\tableofcontents

\section*{Introduction}

Dyson-Schwinger equations are considered in Quantum Field Theory in order to compute the Green functions of the theory as series in the coupling constant.
More precisely, one considers a family of graphs, namely the Feynman diagrams. These graphs are organised in a graded, connected
Hopf algebra; the gradation is given by the number of loops of the graphs. For any primitive Feynman diagram $\gamma$, 
one constructs an insertion operator $B_\gamma$ on this Hopf algebra, and thse operators are used
to define a system of combinatorial equations satisfied by the expansion of the Green functions seen as series in Feynman graphs.

Let us give an example. In QED, we consider three series in Feynman graphs, here denoted by $\photonboson$, $\boson$ and $\photon$,
according to the external structure of the graph appearing in these series. These series satisfy the following system \cite{Yeats}:
\begin{eqnarray*}
\photonboson&=&\sum_{\gamma} B_\gamma\left(\frac{(1+\photonboson)^{1+2l(\gamma)}}
{(1-\photon)^{2l(\gamma)}(1-\boson)^{l(\gamma)}}\right),\\
\boson&=&B_{\graphedeux}\left(\frac{(1+\photonboson)^2}{(1-\photon)^2}\right),\\
 \photon&=&B_{\graphetrois}\left(\frac{(1+\photonboson)^2}{(1-\photon)(1-\boson)}\right).
\end{eqnarray*}
The sum in the first equations run over the set of primitive Feynman graphs with the required external structure.

Such a system has a unique solution. The homogeneous components of this solution generates a subalgebra of the Hopf algebra of Feynman graphs.
An important problem is to know if this subalgebra is Hopf or not; if the answer is affirmative, the next question is to describe this Hopf subalgebra
and to relate it to already known objects. This problem has been answered in the case of a single equation with a unique insertion operator in \cite{Foissy2,Foissy4},
and in the case of systems with a unique insertion operator in each equation in \cite{Foissy3,Foissy6}. These result does not answer the question
for the example described earlier, as an infinite number of insertion operators appears in the first equation. The aim of the present paper is 
to give an answer in the general case. \\

The key point is the fact that, as explained in \cite{Bergbauer,Connes,Kreimer1,Kreimer2},
at least in a convenient quotient of the Hopf algebra of Feynman graphs, the insertion operators satisfy the following $1$-cocycle equation: for all $x$,
$$\Delta \circ B_\gamma(x)=B_\gamma(x)\otimes 1+(Id \otimes B_\gamma)\circ \Delta(x).$$
This allows to replace Feynman graphs by decorated rooted trees and insertion operators by grafting operators, 
with the help of the universal property of the Connes-Kreimer Hopf algebra of rooted trees $\h_{CK}^J$(theorem \ref{1}). 
For example, for the preceding system, we work with trees decorated by the set $J=\{(1,k)\mid k\geq 1\}\cup\{2,3\}$. 
Any element $j$ of $J$ gives rise to a grafting operator $B_j$, consisting of grafting the differents trees of a rooted forest decorated by $J$ on a common root 
decorated by $j$. Using the universal property, we can now consider the system defined on $\h_{CK}^J$:
$$\left\{\begin{array}{rcl}
x_1&=&\displaystyle \sum_{k\geq 1}B_{(1,k)}\left(\frac{(1+x_1)^{1+2k}}{(1-x_2)^k(1-x_3)^{2k}}\right),\\[3mm]
x_2&=&\displaystyle B_2\left(\frac{(1+x_1)^2}{(1-x_3)^2}\right),\\[3mm]
x_3&=&\displaystyle B_3\left(\frac{(1+x_1)^2}{(1-x_2)(1-x_3)}\right).
\end{array}\right.$$
Here are the first terms of the solution:
\begin{eqnarray*}
x_1&=&\tdun{$(1,1)$}\hspace{4mm}+3\tddeux{$(1,1)$}{$(1,1)$}\hspace{4mm}+\tddeux{$(1,1)$}{$2$}\hspace{4mm}
+\tddeux{$(1,1)$}{$3$}\hspace{4mm}+\tdun{$(1,2)$}\hspace{4mm}\\[2mm]
&&+9\tdtroisdeux{$(1,1)$}{$(1,1)$}{$(1,1)$}\hspace{3mm}+3\tdtroisdeux{$(1,1)$}{$(1,1)$}{$2$}\hspace{3mm}
+6\tdtroisdeux{$(1,1)$}{$(1,1)$}{$3$}\hspace{3mm}+2\tdtroisdeux{$(1,1)$}{$2$}{$(1,1)$}\hspace{3mm}
+\tdtroisdeux{$(1,1)$}{$2$}{$3$}\hspace{3mm}+4\tdtroisdeux{$(1,1)$}{$3$}{$(1,1)$}\hspace{3mm}
+2\tdtroisdeux{$(1,1)$}{$3$}{$2$}\hspace{3mm}+2\tdtroisdeux{$(1,1)$}{$3$}{$3$}\hspace{3mm}\\[2mm]
&&+3\hspace{6mm}\tdtroisun{$(1,1)$}{$2$}{\hspace{-5mm}$(1,1)$}\hspace{3mm}
+6\hspace{6mm}\tdtroisun{$(1,1)$}{$3$}{\hspace{-5mm}$(1,1)$}\hspace{2mm}
+\tdtroisun{$(1,1)$}{$2$}{$2$}\hspace{3mm}+2\tdtroisun{$(1,1)$}{$3$}{$2$}\hspace{2mm}
+3\tdtroisun{$(1,1)$}{$3$}{$3$}\hspace{2mm}\\[2mm]
&&+3\tddeux{$(1,1)$}{$(1,2)$}\hspace{4mm}+5\tddeux{$(1,2)$}{$(1,1)$}\hspace{4mm}
+2\tddeux{$(1,2)$}{$2$}\hspace{4mm}+\tddeux{$(1,2)$}{$3$}\hspace{4mm}+\tdun{$(1,3)$}\hspace{4mm}+\ldots\\[5mm]
x_2&=&\tdun{$2$}+2\tddeux{$2$}{$(1,1)$}\hspace{4mm}+\tddeux{$2$}{$3$}\\[2mm]
&&+6\tdtroisdeux{$2$}{$(1,1)$}{$(1,1)$}\hspace{3mm}
+2\tdtroisdeux{$2$}{$(1,1)$}{$2$}\hspace{3mm}+4\tdtroisdeux{$2$}{$(1,1)$}{$3$}\hspace{3mm}
+4\tdtroisdeux{$2$}{$3$}{$(1,1)$}\hspace{3mm}+2\tdtroisdeux{$2$}{$3$}{$2$}+2\tdtroisdeux{$2$}{$3$}{$3$}\\[2mm]
&&+\hspace{6mm}\tdtroisun{$2$}{$(1,1)$}{\hspace{-5mm}$(1,1)$}\hspace{2mm}
+4\hspace{6mm}\tdtroisun{$2$}{$3$}{\hspace{-5mm}$(1,1)$}\hspace{2mm}+3\tdtroisun{$2$}{$3$}{$3$}
+2\tddeux{$2$}{$(1,2)$}\hspace{4mm}+\ldots\\[5mm]
x_3&=&\tdun{$3$}+2\tddeux{$3$}{$(1,1)$}\hspace{4mm}+\tddeux{$3$}{$2$}+\tddeux{$3$}{$3$}\\[2mm]
&&+6\tdtroisdeux{$3$}{$(1,1)$}{$(1,1)$}\hspace{3mm}
+2\tdtroisdeux{$3$}{$(1,1)$}{$2$}\hspace{3mm}+4\tdtroisdeux{$3$}{$(1,1)$}{$3$}\hspace{3mm}
+2\tdtroisdeux{$3$}{$3$}{$(1,1)$}\hspace{3mm}+\tdtroisdeux{$3$}{$3$}{$2$}+\tdtroisdeux{$3$}{$3$}{$3$}\\[2mm]
&&+\hspace{6mm}\tdtroisun{$3$}{$(1,1)$}{\hspace{-5mm}$(1,1)$}\hspace{2mm}
+2\hspace{6mm}\tdtroisun{$3$}{$2$}{\hspace{-5mm}$(1,1)$}\hspace{2mm}
+2\hspace{6mm}\tdtroisun{$3$}{$3$}{\hspace{-5mm}$(1,1)$}\hspace{2mm}
+\tdtroisun{$3$}{$2$}{$2$}+\tdtroisun{$3$}{$3$}{$2$}+\tdtroisun{$3$}{$3$}{$3$}+2\tddeux{$3$}{$(1,2)$}\hspace{4mm}+\ldots
\end{eqnarray*}
The degree of the decorations $(1,1)$, $2$ and $3$ is $1$, the degree of the decoration $(1,2)$ is $2$ and the degree of the decoration $(1,3)$ is $3$.
The universal property allows to construct a Hopf algebra morphism sending $x_1$ to $\photonboson$, $x_2$ to $\boson$, and $x_3$ to $\photon$. \\

Up to a simplification of the hypotheses (see section \ref{s22}), we can now consider without loss of generality systems of the form:
$$(S):\: \forall i\in I,  x_i=\sum_{j\in J_i} B_{(i,j)}\left(f^{(i,j)}(x_k, k\in I)\right),$$
defined on the Hopf algebra $\h_{CK}^J$, the set $J$ being of the form $\displaystyle \bigsqcup_{i\in I} J_i$, where $I$ is a finite set and for all $i\in I$, 
$J_i \subseteq \mathbb{N}^*$; the $f^{(i,j)}$ are formal series.
The grafting operator $B_{(i,j)}$ appearing in this system is homogeneous of degree $j$.\\

We treat in the third section of this text the case of a single equation, that is to say when $|I|=1$. We obtain two possibilities (theorem \ref{12}):
\begin{enumerate}
\item There exists $\lambda,\mu \in K$, such that the equation has the form:
$$x=\sum_{j\in J} B_j((1-\mu x)Q(x)^i),$$
where $Q(x)=(1-\mu x)^{-\lambda/\mu}$ if $\mu \neq 0$ and $Q(x)=e^{\lambda x}$ if $\mu=0$.
\item There exists $m\geq 1$ and $\alpha \in K$ such that the equation has the form:
$$x=\sum_{\substack{j\in J\\ m\mid j}}B_j(1+\alpha x)+\sum_{\substack{j\in J\\ m/ \hspace{-1.2mm} \mid \hspace{1.2mm} j}}B_j(1)$$
\end{enumerate} 

We prove that such an  equation indeed give a Hopf subalgebra $\hs$, and give a description of $\hs$ in the fourth section.
By the Cartier-Quillen-Milnor-Moore theorem \cite{Abe,Milnor}, it is the dual of an enveloping algebra, and it turns out that the underlying Lie algebra $\g$ 
has a complementary structure: it is pre-Lie, that is to say there is a (not necessarily associative) product $\circ$, satisfying for all $x,y,z\in \g$:
$$(x \circ y)\circ z-x\circ (y\circ z)=(y\circ x)\circ z-y\circ (x \circ z).$$
The Lie bracket is the antisymmetrization of $\circ$. This pre-Lie algebra has a basis $(e_i)_{i\geq 1}$ (if in the equation appears
a grafting operator of degree $1$; if not, the set of indices can be strictly included in $\mathbb{N}^*$). 
In the first case, the pre-Lie product is given by $e_i\circ e_j=(\lambda j-\mu)e_{i+j}$. If $\lambda \neq 0$,
this is isomorphic as a Lie algebra to the Faà di Bruno Lie algebra (corollary \ref{20}). Hence, the Hopf subalgebra generated by the solution of the equation 
is isomorphic to the Faà di Bruno Hopf algebra, that is to say the coordinate ring of the group of formal diffeomorphisms tangent to the identity at $0$.
If $\lambda=0$, $\g$ is abelian, and the Hopf subalgebra generated by the solution is isomorphic to the Hopf algebra of symmetric functions.
In the second case, $\circ$ is associative, given by $e_i \circ e_j=\alpha e_{i+j} $ if $j$ is a multiple of $m$ and $0$ otherwise (proposition \ref{21}).\\

We treat the case of systems in the last section. We assume here that in any equation of the system we consider, a grafting operator homogeneous of degree $1$
appears. Then the formal series of the system are entirely determined by the formal series correponding to these grafting operators of degree $1$.
By the classification for Dyson-Schwinger systems with a single operator by equation obtained in \cite{Foissy3}, we can obtain two types of systems,
called  \emph{fundamental} and \emph{quasicyclic} . The description of the other formal series is done in theorem \ref{23} and proposition \ref{25}.

Here is a typical example of a fundamental system (see corollary \ref{24}): here, $I=I_0 \sqcup J_0 \sqcup K_0$, and:
\begin{eqnarray*}
x_i&=&\sum_{q\in J_i} B_{(i,q)}\left((1-\beta_i x_i) \prod_{j\in I_0} (1-\beta_jx_j)^{-\frac{1+\beta_j}{\beta_j}q}
\prod_{j\in J_0}(1-x_j)^{-q}\right), \mbox{ if }i \in I_0,\\
x_i&=&\sum_{q\in J_i} B_{(i,q)}\left((1-x_i) \prod_{j\in I_0} (1-\beta_jx_j)^{-\frac{1+\beta_j}{\beta_j}q}
\prod_{j\in J_0}(1-x_j)^{-q}\right)\mbox{ if }i\in J_0,\\
x_i&=&\sum_{q\in J_i} B_{(i,q)}\left( \prod_{j\in I_0} (1-\beta_jx_j)^{-\frac{1+\beta_j}{\beta_j}q}
\prod_{j\in J_0}(1-x_j)^{-q}\right)\mbox{ if }j\in K_0.
\end{eqnarray*}
Here is a typical example of a quasi-cyclic system: $I=\mathbb{Z}/N\mathbb{Z}$, and for all $\overline{i}\in I$:
$$x_{\overline{i}}=\sum_{j\in J_{\overline{i}}} B_j(1+x_{\overline{i+j}}).$$

{\bf Notations.} Let $K$ be a field of characteristic zero. Any vector space, algebra, Hopf algebra, Lie algebra\ldots of this text will be taken over $K$. \\

{\bf Acknowledgements.} The author would like to thank Karen Yeats and the Simon Fraser University for their hospitality in june 2011.

\section{Recalls}

\label{s1}

\subsection{Hopf algebras of decorated trees}

Let $J$ be a nonempty set. A \emph{rooted tree decorated by $J$} is a couple $(t,d)$, where $t$ is a rooted tree,
that is to say a connected finite graph without loop, with a special vertex called the \emph{root}, and $d$ is a map from
the set of vertices of $t$ into $J$. For example, here are the decorated rooted trees with $k\leq 4$  vertices (the root is the vertex at the bottom of the graph):
$$\tdun{$a$};\: a\in J,\hspace{1cm} \tddeux{$a$}{$b$},\: (a,b)\in J^2;\hspace{1cm}
\tdtroisun{$a$}{$c$}{$b$}=\tdtroisun{$a$}{$b$}{$c$},\: \tdtroisdeux{$a$}{$b$}{$c$},\:(a,b,c)\in J^3;$$
$$ \tdquatreun{$a$}{$d$}{$c$}{$b$}=\tdquatreun{$a$}{$c$}{$d$}{$b$}=\ldots
=\tdquatreun{$a$}{$b$}{$c$}{$d$},\: \tdquatredeux{$a$}{$d$}{$b$}{$c$}=\tdquatretrois{$a$}{$b$}{$c$}{$d$},\:
 \tdquatrequatre{$a$}{$b$}{$d$}{$c$}= \tdquatrequatre{$a$}{$b$}{$c$}{$d$},\: \tdquatrecinq{$a$}{$b$}{$c$}{$d$},\:(a,b,c,d)\in J^4.$$

The Hopf algebra $\h_{CK}^J$ of rooted trees decorated by $J$  \cite{Connes,Kreimer1} is, as an algebra, freely generated by the set of these trees.
As a consequence, a basis of $\h_{CK}^J$ is given by the set of monomials in these trees, which are called \emph{rooted forests decorated by $J$}.
For example, here are the rooted forests with $k\leq 3$ vertices:
$$1,\hspace{1cm}, \tdun{$a$}, a\in J,\hspace{1cm} \tddeux{$a$}{$b$},\: \tdun{$a$}\tdun{$b$}=\tdun{$b$}\tdun{$a$},(a,b)\in J^2,$$
$$\tdun{$a$}\tdun{$b$}\tdun{$c$}=\tdun{$a$}\tdun{$c$}\tdun{$b$}=\ldots=\tdun{$c$}\tdun{$b$}\tdun{$a$},\:
\tdun{$a$}\tddeux{$b$}{$c$}=\tddeux{$b$}{$c$}\tdun{$a$},\:
\tdtroisun{$a$}{$c$}{$b$}=\tdtroisun{$a$}{$b$}{$c$},\:\tdtroisdeux{$a$}{$b$}{$c$},\:(a,b,c)\in J^3.$$

Let $t$ be a rooted tree decorated by $J$. An \emph{admissible cut} of $t$ is a nonempty choice $c$ of edges of $t$ such that any path in $t$ from the root to $t$
to a leaf meets at most one edge in $c$. Deleting these edges, $t$ becomes a forest $W^c(t)$. One of the trees of this forest contains the root of $t$:
it will be denoted by $R^c(t)$. The product of the other trees of $W^c(t)$ is denoted by $P^c(t)$. The coproduct of $\h_{CK}^J$ is then defined for all
tree $t$ by:
$$\Delta(t)=t\otimes 1+1\otimes t+\sum_{\scriptsize \mbox{$c$ admissible cut of $c$}}P^c(t)\otimes R^c(t).$$
Here is an example of coproduct:
$$\Delta(\tdquatredeux{$d$}{$c$}{$b$}{$a$})=\tdquatredeux{$d$}{$c$}{$b$}{$a$} \otimes 1+1\otimes \tdquatredeux{$d$}{$c$}{$b$}{$a$}
+\tddeux{$b$}{$a$} \otimes \tddeux{$d$}{$c$}+\tdun{$a$}\otimes \tdtroisun{$d$}{$c$}{$b$}+
\tdun{$c$} \otimes \tdtroisdeux{$d$}{$b$}{$a$} +\tddeux{$b$}{$a$}\tdun{$c$}\otimes \tdun{$d$}
+\tdun{$a$}\tdun{$c$}\otimes \tddeux{$d$}{$b$}.$$

For all $j \in J$, let $B_j:\h_{CK}^J \longrightarrow \h_{CK}^J$, sending a forest decorated by $J$ to the tree obtained by grafting
the trees of this forest on a common root decorated by $j$. For example, $B_j(\tdun{$k$}\tddeux{$i$}{$l$})=
\tdquatretrois{$j$}{$i$}{$l$}{$k$}$. It is proved in \cite{Connes} (in the non decorated version) that for all $j\in J$, for all $x \in \h_{CK}^J$,
$$\Delta\circ B_j(x)=B_j(x) \otimes 1+(Id \otimes B_j) \circ \Delta(x).$$
In other words, $B_j$ is a $1$-cocycle for the Cartier-Quillen cohomology. Moreover, the following universal property is satisfied:

\begin{theo} \label{1}
Let $A$ be a Hopf algebra and for all $j\in J$, let $L_j:A\longrightarrow A$ be a $1$-cocycle of $A$. There exists a unique Hopf algebra
morphism $\phi:\h_{CK}^J \longrightarrow A$ such that for all $j\in J$, $\phi \circ B_j=L_j \circ \phi$.
\end{theo}

\subsection{Gradation and graded dual}

We now assume that $J$ is a \emph{graded, connected} set, that is to say $J$ is given a map $deg:J \longrightarrow \mathbb{N}^*$. 
for any $j\in J$, $deg(j)$ will be called the \emph{degree} of $j$. In this case, $\h_{CK}^J$ becomes a graded Hopf algebra,
the decorated forests being homogeneous of the degree given by the sum of the degrees of their decorations.
For this gradation, $B_j$ si homogeneous of the same degree as $j$.\\

Let us consider a (convenient quotient of a) Hopf algebra $\h$ of Feynman graphs and let $(\gamma_j)_{j\in J}$ be a family of primitive Feynman graphs
of $\h$. We give a gradation to $J$ by putting $j\in J$ of degree the number of loops of $\gamma_j$. By the universal property, 
there exists a unique Hopf algebra morphism $\phi$ from $\h_{CK}^J$ to $\h$, such that $\phi\circ B_j=B_{\gamma_j} \circ \phi$ for all $j\in J$. 
It is not difficult to show that this morphism is homogeneous of degree $0$. As a consequence, it is possible to lift any systems of Dyson-Schwinger equations using 
the insertion operators $B_{\gamma_j}$ to a system of Dyson-Schwinger equations in $\h_{CK}^J$ using the operators $B_j$. Hence, $\phi$ sends 
the homogeneous components of the solution of these equations in $\h_{CK}^J$ to the homogeneous components of the solution in $\h$. Consequently, 
if these components generate a Hopf subalgebra of $\h_{CK}^J$, it is also the case in $\h$. \\

If for all $n \geq 1$, the number of elements of $J$ of degree $n$ is finite, then the homogeneous components of $\h_{CK}^J$ are finite-dimensional.
So the graded dual $(\h_{CK}^J)^*$ of $\h_{CK}^J$ is also a Hopf algebra. As $(\h_{CK}^J)^*$ is graded, connected, commutative,
$(\h_{CK}^J)^*$ is graded, connected, cocommutative. By the Cartier-Quillen-Milnor-Moore theorem, it is the enveloping algebra of a certain Lie algebra  
$\g_{CK}^J$. To any rooted tree $t$ decorated by $J$, we associate a linear form on $\h_{CK}^J$ also denoted by $t$ by $\langle t,F\rangle=s_t \delta_{t,F}$ 
for any rooted forest $F$ decorated by $J$, where $s_t$ is the number of symmetries of $t$. Then the set of rooted trees decorated by $J$ becomes a basis 
of $\g_{CK}^J$. By similarity with the non-decorated situation of \cite{Connes}, the bracket of $\g_{CK}^J$ is given by:
$$[t,t']=\sum\mbox{grafting of  $t$ over $t'$}-\sum \mbox{graftings of $t'$ over $t$}.$$
This bracket is the antisymmetrization of the product defined by:
$$t\circ t'=\sum\mbox{grafting of  $t$ over $t'$}$$
For example, $\tdun{$a$}\circ \tdtroisun{$b$}{$d$}{$c$}=\tdquatreun{$b$}{$d$}{$c$}{$a$}+\tdquatredeux{$b$}{$d$}{$c$}{$a$}
+\tdquatretrois{$b$}{$d$}{$a$}{$c$}$. This product is not associative, but is (left) pre-Lie, that is to say, for all $x,y,z \in \g_{CK}^J$:
$$(x \circ y) \circ z-x \circ (y\circ z)=(y\circ x)\circ z-y\circ (x \circ z).$$
By the results of \cite{Chapoton2}, $\g_{CK}^J$ is the free pre-Lie algebra generated by the $\tdun{$j$}$'s, $j\in J$.
For more details on the Hopf algebra $(\h_{CK}^J)^*$, see section \ref{s4}. \\

{\bf Remark.} The solutions of Dyson-Schwinger equations are not elements of $\h_{CK}^J$: we have to complete this space, in the following sense.
We consider the case where $J$ is a graded, connected set, in such a way that $\h_{CK}^J$ is a graded, connected Hopf algebra.
The valuation $val$ on $\h_{CK}^J$ associated to this gradation induces a distance on $\h_{CK}^J$ defined by $d(x,y)=2^{-val(x-y)}$.
The space $\h_{CK}^J$ is not complete for this distance; the completion of $\h_{CK}^J$ is the space of formal series in rooted trees decorated by trees.
So elements of this completion can uniquely be written under the form $\sum a_F F$, where the sum runs over all rooted forests $F$ decorated by $J$.

\subsection{The Faà di Bruno Hopf algebra}

Let us consider the group of formal diffeomorphisms of the line tangent to the identity:
$$G_{FdB}=\{x+a_1x^2+a_2x^3+\ldots\:\mid\: \forall i\geq 1,a_i \in K\}.$$
The product of this group is the usual composition of formal series. The Faà di Bruno Hopf algebra $\h_{FdB}$ is the co-opposite of the coordinate ring of 
$G_{FdB}$. As an algebra, it is the free associative, commutative algebra generated by the $x_i$'s, $\geq 1$, where:
$$x_i:\left\{\begin{array}{rcl}
G_{FdB}&\longrightarrow&K\\
x+a_1x^2+\ldots&\longrightarrow&a_i.
\end{array}\right.$$
The coproduct is defined by $\Delta(f)(F_1\otimes F_2)=f(F_2\circ F_1)$, for all $f\in \h_{FdB}$, for all $F_1,F_2\in G_{FdB}$.
For example:
\begin{eqnarray*}
\Delta(x_1)&=&x_1 \otimes 1+1\otimes x_1,\\
\Delta(x_2)&=&x_2 \otimes 1+1\otimes x_2+2x_1 \otimes x_1,\\
\Delta(x_3)&=&x_3 \otimes 1+1\otimes x_3+2x_2 \otimes x_1+3x_1 \otimes x_2+x_1^2 \otimes x_1.
\end{eqnarray*}
It is a graded, connected, commutative Hopf algebra. By the Cartier-Quillen-Milnor-Moore theorem, its dual is an enveloping algebra.
The underlying Lie algebra has a basis $(e_i)_{i\geq 1}$, dual to the $x_i$'s, and the Lie bracket is given by $[e_i,e_j]=(j-i)e_{i+j}$.
One can also define a product on this Lie algebra by $e_i \circ e_j(x_k)=(e_i \otimes e_j) \circ \Delta(x_k)$ for all $i,j,k\geq 1$.
This gives $e_i \circ e_j=(j+1)e_{i+j}$. This product is pre-Lie and induces the Lie bracket by antisymmetrization.

\section{Definitions of Hopf systems}

\subsection{Definitions}

\begin{defi} \begin{enumerate}
\item Let us choose a finite, non-empty set $I$. For any $i\in I$, let $J_i$ be a graded, connectet set.
We put $J=\{(i,j)\:\mid\: i\in I, j\in I_i\}$ and we work in the Hopf algebra $\h_{CK}^J$ of rooted trees decorated by $J$.
It is a graded Hopf algebra, the degree of the decoration $(i,j)$ being the degree of $j$. 
\item For all $i\in I$, for all $j\in J_i$, let $f^{(i,j)} \in K[[h_k,k\in I]]$. The \emph{system of Dyson-Schwinger equations} associated to these data is:
$$(S):\: \forall i\in I,  x_i=\sum_{j\in J_i} B_{(i,j)}\left(f^{(i,j)}(x_k, k\in I)\right).$$
\item This system has a unique solution in the completion of $\h_{CK}^J$, denoted by $x=(x_i)_{i\in I}$. 
For all $i\in I$, for all $n \geq 1$, the homogeneous component of degree $n$ of $x_i$ is denoted by $x_i(n)$. 
If the subalgebra $\hs$ generated by the $x_i(n)$'s is Hopf, we shall say that the system is \emph{Hopf}. 
\end{enumerate} \end{defi}

{\bf Notations.} \begin{enumerate}
\item We shall often take $I=\{1,\ldots,N\}$. For all $(i,q) \in J$, we put:
$$f^{(i,q)}=\sum_{p_1,\ldots,p_N} a^{(i,q)}_{(p_1,\ldots,p_N)} h_1^{p_1}\ldots h_N^{p_N}.$$
\item The coefficient of $h_j$ in $f^{(i,q)}$ is also denoted by $a^{(i,q)}_j$; the coefficient of $h_j h_k$ in $f^{(i,q)}$
is also denoted by $a^{(i,q)}_{j,k}$, and so on. \\
\end{enumerate}

The unique solution of $(S)$ is denoted in the following way: for all $i\in I$, $x_i=\sum a_t t$, where the sum is over all trees with a root decorated
by an element $(i,x)$, with $x \in J_i$. The coefficients $a_t$ are computed inductively:
\begin{itemize}
\item If $t=\tun_{(i,q)}$, $a_t=a^{(i,q)}_{(0,\ldots,0)}$.
\item If $t=B_{(i,q)}\left(t_{1,1}^{p_{1,1}}\ldots t_{1,k_1}^{p_{1,k_1}}\ldots 
t_{N,1}^{p_{N,1}}\ldots t_{N,k_N}^{p_{N,k_N}}\right)$, where the $t_{i,j}$'s are different trees, the root of $t_{i,j}$ being
decorated by an element $(i,x)$ with $x\in J_i$:
$$a_t=a^{(i,q)}_{(p_{1,1}+\ldots+p_{1,k_1},\ldots,p_{N,1}+\ldots+p_{N,k_N})}
\prod_{l=1}^N \frac{(p_{l,1}+\ldots+p_{l,k_l})!}{p_{l,1}!\ldots p_{l,k_l}!}\prod_{j,k} a_{t_{j,k}}^{p_{j,k}}.$$
\end{itemize}

\subsection{Simplification of the hypotheses}

\label{s22} We shall only consider systems with non-zeros $x_i$'s.  If $x_i=0$, we can forget one equation and send $h_i$ to zero in all the formal series which appear,
and this gives a system with a strictly smaller number of equations, giving the same subalgebra. In this case, for all $i\in I$, $x_i$ is a non-zero infinite span
of rooted trees with roots decorated by elements of the form $(i,j)$, $j\in J_i$. Consequently, the $x_i$'s are algebraically independent.

\begin{lemma}
Let $(S)$ be a Hopf SDSE, and let $(i,j) \in J$. If $f^{(i,j)}(0)=0$, then $f^{(i,j)}=0$.
\end{lemma}

\begin{proof} As $f^{(i,j)}(0)=0$, $\tun_{(i,j)} \notin \h_{(S)}$.  Moreover, the term $f^{(i,j)}(x_k,k\in I)\otimes \tun_{(i,j)}$
appears in the coproduct of $x_i$, so is an element of the completion of $\h_{(S)}\otimes \h_{(S)}$. As $\tun_{(i,j)}\notin \h_{(S)}$,
we deduce that $f^{(i,j)}(x_k,k\in I)=0$. As the $x_i$'s are algebraically independent, $f^{(i,j)}=0$. \end{proof}\\

{\bf Remark.} So we shall assume in the sequel that for any $(i,j)\in J$, $f^{(i,j)}(0)\neq 0$. Up to a normalization,
we shall assume that $f^{(i,j)}(0)=1$ for all $i,j \in I$.

\begin{lemma}
Let $(S)$ be a Hopf SDSE. Let us fix $i\in I$. If $j,k \in I_i$ have the same degree, then $f^{(i,j)}=f^{(i,k)}$.
\end{lemma}

\begin{proof} Let $n=deg(j)=deg(k)$.  The coefficients of $\tun_{(i,j)}$ and $\tun_{(i,k)}$ in $x_i(n)$ are both equal to $1$.
So in any element of $\h_{(S)}$, The coefficients of $\tun_{(i,j)}$ and $\tun_{(i,k)}$ in $x_i$ are equal. Let us put:
$$\Delta(x_i)=x_i \otimes 1+\sum y_t \otimes t,$$
where the sum is over all the rooted trees with a root decorated by $(i,j)$, $j\in J_i$. As $\Delta(x_i)$ is in the completion of $\h_{(S)}\otimes \h_{(S)}$,
$y_{\tun_{(i,j)}}=y_{\tun_{(i,k)}}$. Moreover:
$$f^{(i,j)}(x_l,l\in J)=y_{\tun_{(i,j)}}=y_{\tun_{(i,k)}}=f^{(i,k)}(x_l,l\in J).$$
As the $x_i$'s are algebraically independent, $f^{(i,j)}=f^{(i,k)}$.  \end{proof}\\

Hence, as a consequence, if $(S)$ is a Hopf system, it can be written as:
$$\forall i\in I,\: x_i=\sum_{n\geq 1} \underbrace{\left(\sum_{j\in J_i,\:deg(j)=n} B_{(i,j)}\right)}_{B_{(i,n)}}\left(f^{(i,n)}
(x_k, k\in I)\right),$$
where $f^{(i,n)}$ is any $f^{(i,j)}$ such that $deg(j)=n$. \\

Finally, it is enough to consider only the Hopf SDSE such that $J_i \subseteq \mathbb{N}^*$ for all $i\in I$, the degree being the canonical inclusion
of $J_i$ into $\mathbb{N}^*$.

\begin{prop}
Let $(S)$ be a SDSE of the form:
$$(S):\: \forall i\in I,  x_i=\sum_{j\in J_i} B_{(i,j)}\left(f^{(i,j)}(x_k, k\in I)\right),$$
with for all $i\in I$, $1\in J_i \subseteq \mathbb{N}^*$. The \emph{truncation at $1$} of $(S)$ is the SDSE:
$$(S'):\: \forall  i\in I, x'_i=B_{(i,1)}\left(f^{(i,1)}(x'_k, k\in I)\right).$$
If $(S)$ is Hopf, then $(S')$ is also Hopf. \end{prop}

\begin{proof} Let $\phi: \h_{CK}^J \longrightarrow \h_{CK}^{J'}$ being the projection on $\h_{CK}^{J'}$ 
sending any forest with at least a vertex decorated by an element $(i,j)$, $j\neq 1$, to zero. 
It is clearly a Hopf algebra morphism. Moreover, $\phi(x)=x'$, so $\phi(\h_{(S)})=\h_{(S')}$.
As a consequence, if $\h_{(S)}$ is Hopf, $\h_{(S')}$ also is.  \end{proof}

\subsection{Structure constants associated to a Hopf SDSE}

Let $(S)$ be a SDSE. For any decorated rooted tree, we denote by $a_t$ the coefficient of $t$ in the unique $x_i$ where it may appear.

\begin{prop}\label{6}
Let $(S)$ be a SDSE. If it is Hopf, then for any $i,i' \in I$, $q\in J_{i'}$, for any $n\in \mathbb{N}^*$, there exists $\lambda_n^{(i,(i',q))}$
such that for all $t$ of degree $n$, with the root decorated by an element $(i,j)$, $j\in J_i$:
$$\lambda_n^{(i,(i',q))} a_t=\sum_{t'} n_{(i',q)}(t,t')a_{t'},$$
where $n_{(i',q)}$ is the number of leaves $s$ of $t'$ decorated by $(i',q)$ such that the cut of $s$ gives $t$.
\end{prop}

\begin{proof} A basis $\B$ of $\h_{(S)}$ is given by the monomials in the $x_k(p)$'s.
As $\Delta(x_i)$ is an element of the completion of $\h_{(S)}\otimes \h_{(S)}$, it can be written in the basis $\B \otimes \B$.
The unique element of this basis where $\tun_{(i',q)} \otimes t$ appears is $x_{i'}(q) \otimes x_i(n)$.
Let us denote by $\lambda_n^{(i,(i',q))}$ the coefficient of this element in $\Delta(x_i)$. Identifying the coefficient  of $\tun_{(i',q)} \otimes t$
in $\Delta(x_i)$ and $\lambda_n^{(i,(i',q))}x_{i'}(q) \otimes x_i(n)$ in the tensor basis of forests, we obtain:
$$\lambda_n^{(i,(i',q))}a_t=\sum_{t'} n_{(i',q)}(t,t')a_{t'},$$
by definition of the coproduct. \end{proof}\\

{\bf Remark.} The converse is true; the proof uses the fact that the dual Hopf algebra of the Hopf algebra of rooted tree is the enveloping algebra
of a free pre-Lie algebra. This result will not be used here.

\begin{lemma}\label{7}
Let $(S)$ be a Hopf SDSE, such that $1 \in J_i$ for all $i\in I=\{1,\ldots,N\}$. For all $i\in I$, $q\in J_i$, we put:
$$f^{(i,q)}=\sum_{(p_1,\ldots,p_N)} a^{(i,q)}_{(p_1,\ldots,p_N)}h_1^{p_1}\ldots h_N^{p_N}.$$
Then for all $(p_1,\ldots,p_N)\in \mathbb{N}^N$, for all $i,j\in \{1,\ldots,N\}$, for all $q\in I_i$:
$$a^{(i,q)}_{(p_1,\ldots,p_j+1,\ldots,p_N)}=\frac{1}{p_j+1}\left(\lambda_{p_1+\ldots+p_n+q}^{(i,(j,1))}
-\sum_{l=1}^N a^{(l,1)}_j p_l\right)a^{(i,q)}_{(p_1,\ldots,p_N)}.$$
\end{lemma}

\begin{proof} We apply the preceding lemma with $t=B_{(i,q)}\left(\tun_{(1,1)}^{p_1}\ldots \tun_{(N,1)}^{p_N}\right)$. It gives:
\begin{eqnarray*}
\lambda^{(i, (j,1))}_{p_1+\ldots+p_n+q}a_t&=&
(p_j+1)a_{B_{(i,q)}\left(\tun_{(1,1)}^{p_1}\ldots \tun_{(j,1)}^{p_j+1}\ldots \tun_{(N,1)}^{p_N}\right)}\\
&&+\sum_{l=1}^N a_{ B_{(i,q)}\left(\tun_{(1,1)}^{p_1}\ldots  \tun_{(l,1)}^{p_l-1}\ldots \tun_{(N,1)}^{p_N}
\tdeux_{(l,1)}^{(j,1)}\right)}.
\end{eqnarray*}
Computing the different coefficients $a_t$ appearing in this formula, we obtain immediately the result. \end{proof}\\

{\bf Remarks.} \begin{enumerate}
\item As a consequence, if $a^{(i,q)}_{(p_1,\ldots,p_N)}=0$, then 
$a^{(i,q)}_{(p'_1,\ldots,p'_N)}=0$ if for all $n$, $p'_n \geq p_n$. In particular, if $f^{(i,q)}$ is not constant, there exists
$j\in J$ such that $a^{(i,q)}_j\neq 0$.
\item Using the results of \cite{Foissy3}, if for all $i\in J$, $1\in I_i$ and there are no constant $f^{(i,1)}$,
the SDSE $x_i=B_i\left(f^{(i,1)}(x_i)\right)$ for all $i\in I$ generates a Hopf subalgebra and this allows to compute the coefficients 
$\lambda_n^{(i,(j,1))}$ from the coefficients $a^{(i,1)}_j$ and $a^{(i,1)}_{j,k}$ (they are the coefficients $\lambda_n{(i,j)}$ of \cite{Foissy3}). 
Consequently, the formal series $f^{(i,1)}$ determine uniquely all the formal series $f^{(i,q)}$.
\item It is possible to prove that $\lambda_n^{(i,(j,q))}$ does not depend on $q$. This will not be used in the sequel.
\end{enumerate}

\section{Case of a single equation}

We here treat the case of a single equation, that is to say that $I$ is reduced to a single element.
As a consequence, the indices $i$ are not needed, as they are all equal. The equation shall now be written as:
$$x=\sum_{j\in J} B_j\left(f^{(j)}(x)\right),$$
where $J\subseteq \mathbb{N}^*$ and for all $j\in J$, $f^{(j)}(0)=1$. We shall also write $\lambda^{(j)}_n$ instead of $\lambda_n^{(i,(i,j))}$, 
for any $j\in J$. We put, for all $j\in J$:
$$f^{(j)}=\sum_{n=0}^\infty a^{(j)}_n h^n.$$
The unique solution can be written as $x=\sum a_t t$, where the sum is over all trees decorated by $J$, the coefficients $a_t$ being inductively computed as follows:
\begin{itemize}
\item $a_{\tdun{$j$}}=1$ for all $j\in J$.
\item If $t=B_j(t_1^{p_1}\ldots t_k^{p_k})$, where $t_1,\ldots,t_k$ are different trees, then:
$$a_t=a^{(j)}_{p_1+\ldots+p_k} \frac{(p_1+\ldots+p_k)!}{p_1!\ldots p_k!} a_{t_1}^{p_1}\ldots a_{t_k}^{p_k}.$$
\end{itemize}

\subsection{Non constant formal series}

\begin{lemma}\label{8}
Let us consider $i\in J$, such that $f^{(i)}$ is non constant. There exists $\alpha_i,\beta_i \in K$, with $\alpha_i\neq 0$, such that for all $j \in J$, for all $n\geq 1$:
$$\lambda_{ni}^{(j)}=\alpha_i(1+(1+\beta_i)(n-1)).$$
Moreover:
$$f^{(i)}=\sum_{k=0}^\infty \frac{\alpha_i^k(1+\beta_i)\ldots(1+(k-1)\beta_i)}{k!}  h^k=
\left\{\begin{array}{c}
e^{\alpha_i h} \mbox{ if }\beta_i=0,\\[2mm]
(1-\alpha_i \beta_i h)^{-1/\beta_i} \mbox{ if }\beta_i \neq 0.
\end{array}\right.$$
\end{lemma}

\begin{proof} Let us apply proposition \ref{6} with $t=B_i(\tdun{$i$}^n)$ and $j=i$. The only trees $t'$ such that $n_i(t,t')\neq 0$ are $B_i(\tdun{$i$}^{n+1})$ and $B_i(\tdun{$i$}^{n-1}\tddeux{$i$}{$i$})$, so:
$$\lambda_{i(n+1)}^ia^{(i)}_n=(n+1)a^{(i)}_{n+1}+na^{(i)}_1a^{(i)}_n.$$
Equivalently:
\begin{equation}
\label{EQ0} a^{(i)}_{n+1}=\frac{1}{n+1}\left(\lambda_{i(n+1)}^i-n a^{(i)}_1\right) a^{(i)}_n.
\end{equation}
Hence, if $a^{(i)}_1=0$, an easy induction proves that $a^{(i)}_n=0$ for all $n\geq 1$, so $f^{(i)}$ is constant: this is a contradiction.
So $a^{(i)}_1 \neq 0$. \\

Let us now apply proposition \ref{6} with $t=B_i^n(1)$, that is to say  the ladder with $n$ vertices all decorated by $i$. The trees $t'$ such that $n_j(t,t')\neq 0$ 
are $B_i^n \circ B_j(1)$, $B_i^{n-1}(\tdun{$i$}\tdun{$j$})$ and $B_i^k (\tdun{$j$} B_i^{n-k}(1))$, $1\leq k \leq n-2$. So:
$$\lambda_{ni}^{(j)} \left(a^{(i)}_1\right)^{n-1}=\left(a^{(i)}_1\right)^n+2(n-1)\left(a^{(i)}_1\right)^{n-2}a_2^{(i)}.$$
As $a^{(i)}_1\neq 0$, $\lambda_{ni}^{(j)}=a^{(i)}_1+2\frac{a^{(i)}_2}{a^{(i)}_1}(n-1)$.
We then take $\alpha_i=a^{(i)}_1$ and $\beta_i=\frac{2a_2^{(i)}}{\left(a^{(i)}_1\right)^2}-1$, and 
the assertion on $\lambda_{ni}^{(j)}$ is now proved. Replacing in (\ref{EQ0}), we obtain for all $n \geq 1$:
$$a^{(i)}_{n+1}=\frac{1}{n+1}\alpha_i(1+\beta_i n)a^{(i)}_n.$$
The formula for the coefficients of $f^{(i)}$ is the proved by an easy induction. \end{proof}

\begin{lemma}
There exists $\lambda,\mu \in K$, such that if $f^{(i)}$ is non constant, then $\alpha_i=\lambda i-\mu\neq 0$
and $\beta_i=\frac{\mu}{\lambda i-\mu}$.
\end{lemma}

\begin{proof} We denote by $J'$ be the set of indices $i\in J$ such that $f^{(j)}$ is non constant. Let $i,j\in J'$. 
Let us compute $\lambda_{nij}^{(j)}$ in two different ways:
\begin{eqnarray*}
\lambda_{nij}^{(j)}&=&\lambda_{(nj)i}^{(j)}\\
&=&\alpha_i(1+(1+\beta_i)(nj-1))\\
&=&n j\alpha_i(1+\beta_i)-\alpha_i \beta_i,\\
&=&\lambda_{(ni)j}^{(j)}\\
&=&ni\alpha_j(1+\beta_j)-\alpha_j\beta_j.
\end{eqnarray*}
As this is true for all $n \geq 1$, we deduce that $\alpha_i\beta_i=\alpha_j\beta_j$ and $j \alpha_i(1+\beta_i)=i \alpha_j(1+\beta_j)$ for all $i,j \in J'$.
From the first equality, we deduce that there exists $\mu\in K$, such that $\alpha_i\beta_i=\mu$ for all $i\in J'$.
The second equality implies that the vectors $(\alpha_i(1+\beta_i))_{i\in J'}$
and $(i)_{i\in J'}$ are colinear, so there exists $\lambda \in K$, such that $\alpha_i(1+\beta_i)=\lambda i$ for all $i \in J'$.
Hence, $\alpha_i+\mu=\lambda i\neq 0$ as $f_i$ is not constant, and $\beta_i=\mu/\alpha_i$. \end{proof}\\

Let us sum up these results. If $i\in J$, such that $f^{(i)}$ is not constant, then:
$$f^{(i)}=\left\{\begin{array}{l}
(1-\mu h)^{-\frac{\lambda i}{\mu} +1}\mbox{ if }\mu \neq 0,\\[2mm]
e^{\lambda i h} \mbox{ if }\mu=0.
\end{array}\right.$$
This gives:

\begin{prop}\label{10}
Let $(E)$ be a Hopf Dyson-Schwinger equation. Then $J$ can be written as $J=J'\sqcup J''$, and there exists $\lambda, \mu\in K$, $\lambda \neq 0$, such that if we put:
$$Q(h)=\left\{\begin{array}{l}
(1-\mu h)^{-\frac{\lambda}{\mu}} \mbox{ if }\mu \neq 0,\\[2mm]
e^{\lambda h} \mbox{ if }\mu=0,
\end{array}\right.$$
then:
$$(E): x=\sum_{j\in J'} B_j\left((1-\mu x)Q(x)^i\right)+\sum_{j\in J''} B_j(1).$$
\end{prop}

\subsection{Constant formal series in Dyson-Schwinger equations}

We first treat three particular cases.

\begin{lemma}\label{11}
\begin{enumerate}
\item Let us consider a Dyson-Schwinger equation of the form:
$$x=B_i(1)+B_j(f(x)),$$
with $f$ non constant. If it is Hopf, then there exists a non-zero $\alpha\in K$,
 such that $f(h)=1+\alpha h$ or $f(h)=\left(1-\alpha \frac{i}{j-i}h\right)^{\frac{i-j}{i}}$.
\item Let us consider a Dyson-Schwinger equation of the form:
$$x=B_i(1)+B_j(f(x))+B_k(g(x)),$$
with $f,g$ non constant. If it is Hopf, then there exists a non-zero $\alpha\in K$, such that ($f=(1-\alpha i h)^{-\frac{j}{i}+1}$ 
and $g=(1-\alpha i h)^{-\frac{k}{i}+1}$)or ($f=g=1+\alpha h$).
\item Let us consider a Dyson-Schwinger equation of the form:
$$x=B_i(1)+B_j(1)+B_k(f(x)),$$
where $f$ is non constant. Then there exists a non-zero $\alpha\in K$, such that $f=1+\alpha h$.
\end{enumerate}\end{lemma}

\begin{proof} 1.  From lemma \ref{8}, there exists $\alpha=\alpha_j,\beta=\beta_j \in K$, such that for all $n\geq 1$:
$$\lambda_{nj}^{(i)}=\lambda_{nj}^{(j)}=\alpha(1+(1+\beta)(n-1))=\alpha(1+\beta)n-\alpha \beta.$$
Moreover, $f=(1-\alpha \beta h)^{-\frac{1}{\beta}}$ if $\beta$ is not equal to $0$ and $e^{\alpha h}$ if $\beta=0$. \\

We define inductively a family of trees by $t_1=\tddeux{$j$}{$i$}$ and $t_{n+1}=B_j (\tdun{$i$} t_n)$ for all $n\geq 1$.
For example, $t_2=\tdquatretrois{$j$}{$j$}{$i$}{$i$}$. For all $n\geq 1$, $t_n$ is a tree with $n$ vertices decorated by $i$ and $n$ vertices decorated by $j$.
Applying proposition \ref{6} to $t_n$, we obtain: 
$$\lambda_{n(i+j)}^{(i)}(1+\beta)^{n-1}=(n-1)(1+2\beta)(1+\beta)^{n-1}+(1+\beta)^n.$$
Let us assume that $\beta \neq -1$. Then $\lambda_{n(i+j)}^{(k)}=(n-1)(1+2\beta)+1+\beta=n(1+2\beta)-\beta$.
We now compute $\lambda_{j(i+j)}^{(k)}$ in two different ways:
\begin{eqnarray*}
\lambda_{j(i+j)}^{(i)}&=&\lambda_{(i+j)j}^{(i)}\\
&=&\alpha(1+\beta)(i+j)-\alpha\beta,\\
&=&\lambda_{j(i+j)}^{(i)}\\
&=&\alpha j(1+2\beta)-\alpha \beta.
\end{eqnarray*}
Hence, $(1+\beta)(i+j)=j(1+2\beta)$, so $\beta=\frac{i}{j-i}$.
As a conclusion, $\beta=-1$ or $\frac{i}{j-i}$, therefore $f(h)=1+\alpha h$ or $\left(1-\alpha \frac{i}{j-i}h\right)^{\frac{i-j}{i}}$. \\

2. Restricting to $i$ and $j$ (that is to say sending all the forests of $\h_{CK}^J$ with at least one vertex not decorated by $i$ or $j$ to zero), 
from the first point, $f=1+\alpha h$ or $\left(1-\alpha \frac{i}{j-i}h\right)^{\frac{i-j}{i}}$; 
restricting to $i$ and $k$, $g=1+\alpha' h$ or $\left(1-\alpha' \frac{i}{k-i}h\right)^{\frac{i-k}{i}}$.

Let us now restrict to $j$ and $k$, using proposition \ref{10}: $\mu$ cannot be equal to zero, so $f=(1-\mu h)^{-\lambda \frac{j}{\mu}+1}$ 
and $g=(1-\mu h)^{-\lambda \frac{k}{\mu}+1}$. Identifying the two expressions of $f$ and $g$, we obtain two possibilities:
\begin{itemize}
\item $-\lambda \frac{j}{\mu}+1=1$ or  $-\lambda \frac{k}{\mu}+1=1$. Then $\lambda=0$ and $f=g=1-\mu h$.
\item $-\lambda \frac{j}{\mu}+1=\frac{i-j}{i}$ and $-\lambda \frac{k}{\mu} +1=\frac{i-k}{i}$. As $j\neq k$, 
this implies that $\frac{\lambda}{\mu}=1/i$. So $\mu=\lambda i$, $f=(1-\lambda i h)^{-\frac{j}{i}+1}$ and $g=(1-\lambda ih)^{-\frac{k}{i}+1}$.
\end{itemize}

3. Let us restrict to $i$ and $k$. From the first point, $f(h)=1+\alpha h$ or $f(h)=\left(1-\alpha \frac{i}{k-i}h\right)^{\frac{i-k}{i}}$.
We then restrict to $j$ and $k$ and we obtain $f(h)=1+\alpha h$ or $f(h)=\left(1-\alpha \frac{j}{k-j}h\right)^{\frac{j-k}{j}}$.
As $\frac{i-k}{i}\neq \frac{j-k}{j}$, necessarily $f=1+\alpha h$. \end{proof}

\begin{theo}\label{12}
Let $(E)$ be a Hopf Dyson-Schwinger equation of the form:
$$x=\sum_{j\in J} B_j\left(f^{(j)}(x)\right),$$
where $J\subseteq \mathbb{N}^*$ and $f^{(j)}(0)=1$ for all $j\in J$. Then one of the following assertions holds:
\begin{enumerate}
\item  there exists $\lambda, \mu\in K$ such that, if we put:
$$Q(h)=\left\{\begin{array}{l}
(1-\mu h)^{-\frac{\lambda}{\mu}} \mbox{ if }\mu \neq 0,\\
e^{\lambda h} \mbox{ if }\mu=0,
\end{array}\right.$$
then:
$$(E): x=\sum_{j\in J} B_j\left((1-\mu x)Q(x)^j\right).$$
\item There exists $m\geq 0$ and $\alpha \in K-\{0\}$ such that:
$$(E):x=\sum_{\substack{j\in J\\ m\mid j}}B_j(1+\alpha x)+
\sum_{\substack{j\in J\\ m/ \hspace{-1.2mm} \mid \hspace{1.2mm} j}}B_j(1)$$
\end{enumerate}
\end{theo}

\begin{proof} From proposition \ref{10}, we can write:
 $$(E): x=\sum_{j\in J'} B_j\left((1-\mu x)Q(x)^i\right)+\sum_{j\in J''} B_j(1).$$
 If $J''=\emptyset$, we obtain the first case. Let us assume that $J'' \neq \emptyset$.
 If it contains at least two elements $i$ and $j$, then restricting to $i,j$ and any $k\in J'$ we deduce from the third point of lemma \ref{11}
 that $f^{(k)}=1+\alpha_k h$ for any $k\in J'$, where $\alpha_k\in K$. Restricting then to $i$ and any $k,l \in J'$, the second point of lemma \ref{11}
 implies that $\alpha_k=\alpha_l$. So we are reduced to:
 \begin{equation}
\label{E1}(E):x=\sum_{j\in J'} B_j(1+\alpha x)+\sum_{j\in J''} B_j(1).
\end{equation}
 If $J'$ contains a unique element $i$, restricting to $i$ and $j,k\in J'=J-\{i\}$, the second point of lemma \ref{11} implies that there are two possibilities:
 \begin{itemize}
\item  First case:
$$(E):x=\sum_{j\in J-\{i\}}B_j\left((1-\alpha i x)^{-\frac{j}{i}+1}\right)+B_i(1).$$
 Noticing that $-j/i+1=0$ if $j=i$, this is the first case, with $\mu=\alpha i=\lambda i$.
 \item Second case:
$$(E)=\sum_{j\in J-\{i\}}B_j(1+\alpha x)+B_i(1).$$
 This is an equation of the form (\ref{E1}).
\end{itemize}
It remains now to consider the case of an equation of the form (\ref{E1}).  Then:
$$x=\underbrace{\sum_{k=1}^{\infty}\sum_{i_1,\ldots,i_k\in J'} \alpha^{k-1}B_{i_1}\circ \ldots \circ B_{i_k}(1)}_{x'}
+\underbrace{\sum_{k=1}^{\infty}\sum_{i_1,\ldots,i_{k-1}\in J',i_k\in J''} \alpha^{k-1}B_{i_1}\circ \ldots \circ B_{i_k}(1)}_{x''}.$$
It is not difficult to see that $\Delta(x')=x'\otimes 1+1\otimes x'+\alpha x' \otimes x'$ and $\Delta(x'')=x''\otimes 1+1\otimes x''+\alpha x'' \otimes x'$, so:
$$\Delta(x)=x\otimes 1+1\otimes x+\alpha x \otimes x'.$$
 Finally, taking the homogeneous component of degree $n$:
$$\Delta(x(n))=x(n)\otimes 1+1\otimes x(n)+\sum_{k=1}^{n-1}\alpha x(n-k) \otimes x'(k).$$
As a consequence, the equation $(E)$ is Hopf, if and only if, for all $k \geq 1$, $x'(k)$ is colinear to $x(k)$.
As $x(k)=x'(k)+x''(k)$ and the ladders which may appear in $x'$ and $x''$ are different, this is equivalent to the following assertion: 
for all $k\geq 1$, $x'(k)=0$  or $x''(k)=0$.

Let us consider the subgroup of $(\mathbb{Z},+)$ generated by the elements of $J'$. it is equal to $m\mathbb{Z}$ for a well-chosen $m\geq 0$.
All the elements of $J'$ are multiples of $m$. Let us assume that there exists $j \in J''$, such that $m \mid j$. By definition of $m$, there exists $j_1,\ldots,j_n \in J'$, 
$\lambda_1,\ldots,\lambda_n\in \mathbb{Z}$, such that $\lambda_1j_1+\ldots+\lambda_nj_n=j$. Writing this equality in a different way,
there exists $i_1,\ldots,i_k,i'_1,\ldots,i'_l \in J'$, $\mu_1,\ldots,\mu_k,\mu'_1,\ldots,\mu'_l>0$, such that
$\mu_1i_1+\ldots+\mu_ki_k=\mu'_1i'_1+\ldots+\mu'_li_l+j$.  So any ladder $l'$ with $\mu_1$ vertices decorated by $i_1$, $\ldots$, $\mu_k$ vertices 
decorated by $i_k$  and any ladder $l''$ with $\mu'_1$ vertices decorated by $i'_1$, $\ldots$, $\mu'_l$ vertices decorated by $i'_l$ 
and its leaf decorated by $j$ have the same degree $m$. So $x'(m)$ and $x''(m)$ are both non zero, and $(E)$ is not Hopf. 
Hence, $m$ divides no $j\in J''$. Finally, $(E)$ can be written as in the second case. \end{proof}\\

{\bf Remarks.} \begin{enumerate}
\item In the first case, for any values of $\lambda$ and $\mu$:
$$\left\{\begin{array}{rcl}
Q(h)&=&\displaystyle \sum_{n=0}^\infty \frac{\lambda (\lambda+\mu)\ldots(\lambda+(n-1)\mu)}{n!}h^n,\\[6mm]
f^{(j)}(x)&=&\displaystyle \sum_{n=0}^\infty \frac{(\lambda j -\mu)(\lambda j)(\lambda j+\mu)\ldots(\lambda j+\mu(n-2))}{n!}h^n.
\end{array}\right.$$
\item In the second case, if $m$ divides $n$, then $x''(n)=0$; if $m$ does not divide $n$, then $x'(n)=0$.
From the preceding proof, the second case gives indeed a Hopf subalgebra. Moreover, for all $n$:
\begin{equation}
\label{E2}\Delta(x(n))=x(n)\otimes 1+1\otimes x(n)+\alpha \sum_{\substack{k=1\\m\mid k}\\ }^{n-1} x(n-k) \otimes x(k).
\end{equation}
\item The first and second cases are not disjoint. A first case with $\lambda=0$ is also a second case with $m=1$.
\end{enumerate}

\section{Pre-Lie structures associated to Hopf Dyson-Schwinger equations}

\label{s4}  We now prove that the equations of theorem \ref{12} are Hopf.

\subsection{The Faà di Bruno pre-Lie algebra}

Let $\g_{FdB}=Vect(e_i\:\mid \: i\geq 1)$ and let $\lambda,\mu \in K$. One defines a pre-Lie product on $\g_{FdB}$ 
by $e_i \circ e_j=(\lambda j-\mu)e_{i+j}$.  It is graded, $e_i$ being homogeneous of degree $i$ for all $i\geq 1$.\\

{\bf Remarks.} \begin{enumerate}
\item The associated Lie bracket is given by $[e_i,e_j]=\lambda(j-i)e_{i+j}$, so does not depend of $\mu$. 
\item If $\lambda=1$ and $\mu=-1$, this pre-Lie algebra is precisely the dual Lie algebra of $\h_{FdB}$.
If $\lambda\neq 0$ and $\mu=-\lambda$, these two pre-Lie algebras are isomorphic; if $\lambda \neq 0$, they are isomorphic as Lie algebras only. In this case,
the graded dual of the enveloping algebra of $\g_{FdB}$ is isomorphic to the Faà di Bruno Hopf algebra $\h_{FdB}$. 
\end{enumerate}
 
We shall use the following result:

\begin{theo} \cite{Schedler,Oudom}
Let $(\g,\circ)$ a pre-Lie algebra. Let $S_+(\g)$ the augmentation ideal of $S(\g)$.
One can extend the product $\circ$ to $S_+(\g)$ in the following way: if $a,b,c \in S_+(\g)$, $x\in \g$,
$$\left\{\begin{array}{rcl}
a\circ 1&=&\varepsilon(a),\\
1\circ b&=&b,\\
(xa)\circ b &=&x \circ (a \circ b)-(x \circ a)\circ b,\\
a \circ (bc)&=&\sum (a' \circ b)(a'' \circ c).\end{array}\right.$$
We define a product on $S_+(\g)$ by $a\star b=\sum a' (a'' \circ b)$, with the Sweedler notation $\Delta(a)=\sum a' \otimes a''$.
This product is extended to $S(\g)$, making $1$ the unit of $\star$.
With its usual coproduct, $S(\g)$ is a Hopf algebra, isomorphic to $\U(\g)$ via the isomorphism:
$$\Phi_\g:\left\{\begin{array}{rcl}
\U(\g)&\longrightarrow&(S(\g),\star)\\
v\in \g&\longrightarrow&v.
\end{array}\right.$$
\end{theo}

Let us start by $\g_{CK}^J$. Here, a basis of $S(\g_{CK}^J)$ is given by the set of rooted forests decorated by $J$.

\begin{prop}
Let $F=t_1\ldots t_n , G$ be two decorated forests. Then:
$$F \circ G=\sum_{s_1,\ldots,s_n \in G} \mbox{graftings of $t_1$ over $s_1$,$\ldots $, $t_n$ over $s_n$}.$$
\end{prop}

\begin{proof} By induction on $n$. Let us start with $n=1$. We put $G=s_1\ldots s_m$ and we proceed inductively on $m$.
If $m=1$, it is the definition of $\circ$ on $\g_{CK}^J$. Let us assume the result at rank $m-1$. We put $G'=s_1\ldots s_{m-1}$.
Then:
\begin{eqnarray*}
t_1 \circ G&=&t_1 \circ (G's_m)\\
&=&(t_1 \circ G')s_m+G'(t_1 \circ s_m)\\
&=&\sum_{s\in G'}(\mbox{grafting of $t_1$ over $s$})s_m+\sum_{s \in s_m}G'(\mbox{grafting of $t_1$ over $s$})\\
&=&\sum_{s\in G} \mbox{grafting of $t_1$ over $s$}.
\end{eqnarray*}
So the result is true at rank $1$. Let us assume it at rank $n-1$. We put $F'=t_2\ldots t_n$. Then:
\begin{eqnarray*}
F \circ G&=&t_1 \circ (F' \circ G)-(t_1 \circ F)\circ G\\
&=&\sum_{s_2,\ldots,s_n \in G} \sum_{s \in F' \cup G}\mbox{grafting of $t_2$ over $s_2$,$\ldots$, $t_n$ over $s_n$, $t_1$ over $s$}\\
&&-\sum_{s_2,\ldots,s_n \in G} \sum_{s \in F'}\mbox{grafting of $t_2$ over $s_2$,$\ldots$, $t_n$ over $s_n$, $t_1$ over $s$}\\
&=&\sum_{s_2,\ldots,s_n \in G} \sum_{s \in G}\mbox{grafting of $t_2$ over $s_2$,$\ldots$, $t_n$ over $s_n$, $t_1$ over $s$}\\
&=&\sum_{s_1,\ldots,s_n \in G}\mbox{grafting of $t_1$ over $s_1$,$\ldots$, $t_n$ over $s_n$}.
\end{eqnarray*}
So the result is true for all $n$. \end{proof}

\begin{cor}
If $F=t_1\ldots t_m$ and $G$ are decorated forests, then:
$$F \star G=\sum_{k=0}^m \sum_{1\leq i_1<\ldots<i_k \leq m} \sum_{s_1,\ldots,s_m \in G}
(\mbox{grafting of $t_1$ over $s_1$, $\ldots $, $t_k$ over $s_k$}) \prod_{i\neq i_1,\ldots,i_k} t_i.$$
\end{cor}

This Hopf algebra is known as the Grossman-Larson Hopf algebra \cite{Grossman1,Grossman2}.  
Extending the pairing between trees and forests defined in section \ref{s1}, it is isomorphic to the graded dual of $\h_{CK}^J$, via the pairing:
$$\langle-,-\rangle: \left\{\begin{array}{rcl}
S(\g_{CK}^J)\otimes \h_{CK}^J&\longrightarrow&K\\
(F,G)&\longrightarrow&s_F \delta_{F,G},
\end{array}\right.$$
where $F,G$ are two forests and $s_F$ is the number of symmetries of $F$ \cite{Hoffman,Panaite}. \\

Let us now consider the Faà di Bruno pre-Lie algebra.

\begin{prop}
In $S(\g_{FdB})$:
$$(e_{i_1}\ldots e_{i_m})\circ e_j=(\lambda j-\mu)(\lambda j) (\lambda j+\mu)\ldots (\lambda j+(m-2)\mu) e_{i_1+\ldots+i_m+j}.$$
\end{prop}

\begin{proof} We put $P_m(j)=(\lambda j-\mu)(\lambda j) (\lambda j+\mu)\ldots (\lambda j+(m-2)\mu)$.
We proceed inductively on $m$. If $m=1$, it is the definition of the pre-Lie product of $\g_{FdB}$.
Let us assume the result at rank $m-1$. Then:
\begin{eqnarray*}
(e_{i_1}\ldots e_{i_m})\circ e_j&=&e_{i_1}\circ ((e_{i_2}\ldots e_{i_m})\circ e_j)
-(e_{i_1}\circ (e_{i_2}\ldots e_{i_m})) \circ e_j\\
&=&P_{m-1}(j)e_{i_1} \circ  e_{i_2+\ldots+i_m+j}
-\sum_{k=2}^m (\lambda i_k-\mu)(e_{i_2}\ldots e_{i_1+i_k}\ldots e_{i_m})\circ e_j\\
&=&P_{m-1}(j)(\lambda(i_2+\ldots+i_m+j)-\mu)e_{i_1+\ldots+i_m+j}\\
&&-\sum_{k=2}^m  P_{m-1}(j)(\lambda i_k-\mu)e_{i_1+\ldots+i_m+j}\\
&=&P_{m-1}(j)(\lambda(i_2+\ldots+i_m+j-i_2-\ldots-i_m)-\mu+(m-1)\mu)e_{i_1+\ldots+i_m+j}\\
&=&P_m(j)e_{i_1+\ldots+i_m+j}.
\end{eqnarray*}
So the result is true for all $n$. \end{proof}

\subsection{Morphisms of pre-Lie algebras}

Let us choose a set $J\subseteq \mathbb{N}^*$. As $\g_{CK}^J$ is the free pre-Lie algebra generated by the $\tdun{$j$}$'s,
there is a unique pre-Lie algebra morphism:
$$\phi:\left\{\begin{array}{rcl}
\g_{CK}^J&\longrightarrow&\g_{FdB}\\
\tdun{$j$},\:j\in J&\longrightarrow& e_j.
\end{array}\right.$$

\begin{lemma}
For all decorated rooted tree $t$, $\exists \mu_t\in K$, $\phi(t)=\mu_t e_{|t|}$. the coefficients $\mu_t$ can be inductively computed in the following way:
\begin{enumerate}
\item $\mu_{\tdun{$j$}}=1$.
\item Let $t=B_j(t_1\ldots t_m)$, with $m\geq 1$. Then:
$$\mu_t=\mu_{t_1}\ldots \mu_{t_m}(\lambda j-\mu)(\lambda j)(\lambda j+\mu)\ldots (\lambda j+(m-2)\mu).$$
\end{enumerate}\end{lemma}

\begin{proof} We extend $\phi$ in a Hopf algebra morphism from $S(\g_{CK}^J)$ to $S(\g_{FdB})$. Then,
$\phi(a\circ b)=\phi(a)\circ \phi(b)$ for any $a,b \in S(\g_{CK}^J)$. The first point is obvious. For the second point:
\begin{eqnarray*}
\phi(t)&=&\phi((t_1\ldots t_m) \circ \tdun{$j$})\\
&=&(\phi(t_1)\ldots \phi(t_m) )\circ \phi(\tdun{$j$})\\
&=&\mu_{t_1}\ldots \mu_{t_m} e_{|t_1|}\ldots e_{|t_m|} \circ e_j\\
&=&\mu_{t_1}\ldots \mu_{t_m}(\lambda j-\mu)(\lambda j)(\lambda j+\mu)\ldots (\lambda j+(m-2)\mu)e_{|t_1|+\ldots+|t_m|+j}\\
&=&\mu_{t_1}\ldots \mu_{t_m}(\lambda j-\mu)(\lambda j)(\lambda j+\mu)\ldots (\lambda j+(m-2)\mu)e_{|t|}.
\end{eqnarray*}
So the announced result holds. \end{proof}

\begin{lemma}\begin{enumerate}
\item If $\lambda \neq 0$, the morphism $\phi$ is surjective if, and only if, ($1\in J$) and ($2\in J$ or $\mu \neq \lambda$).
\item If $\lambda=0$, the morphism $\phi$ is surjective if, and only if,  ($\mu\neq 0$ and $1 \in J$) or ($J=\mathbb{N}^*$).
\end{enumerate}\end{lemma}

\begin{proof} There is a unique tree decorated by $\mathbb{N}^*$ of degree $1$ , which is $\tdun{$1$}$.
By homogeneity, $Im(\phi)_1=(0)$ if $1 \notin J$. So if $\phi$ is surjective, $1\in J$.

There are two trees decorated by $\mathbb{N}^*$ of degree $2$, which are $\tdun{$2$}$ and $\tddeux{$1$}{$1$}$.
Moreover, $\phi(\tddeux{$1$}{$1$})=(\lambda-\mu)e_2$. Therefore, if $2 \notin J$ and $\lambda=\mu$, then $Im(\phi)_2=(0)$.
As a consequence, if $\phi$ is surjective, $\mu \neq \lambda$ or $2 \in J$. 

If $\lambda=\mu=0$, then $\phi(t)=0$ if $t$ is not reduced to a single root, and $\phi(\tdun{$i$})=e_i$ for all $i\in J$. 
So $Im(\phi)=Vect(e_i\mid i\in J)$. As a consequence, if $\lambda=\mu=0$, then $\phi$ is surjective if, and only if, $j=\mathbb{N}^*$.\\

These three observations prove $\Longrightarrow$ in both cases. \\

1. $\Longleftarrow$. Let us first assume that $1\in J$ and $\mu \neq \lambda$. Let us prove by induction that $e_n \in Im(\phi)$. This is obvious for $n=1$.
If $e_{n-1}\in Im(\phi)$, then $e_{n-1} \circ \phi(\tdun{$1$})=e_{n-1}\circ e_1=(\lambda-\mu)e_n \in Im(\phi)$, so $e_n \in Im(\phi)$.
Hence, $\phi$ is surjective.

Let us assume that $1,2 \in J$ and $\mu=\lambda$. Let us prove by induction that $e_n \in Im(\phi)$. This is obvious for $n=1,2$.
If $e_1,\ldots,e_{n-1}\in Im(\phi)$, then $e_{n-2}\circ \phi(\tdun{$2$})=e_{n-2}\circ e_2=\lambda e_n$,
so $e_n \in Im(\phi)$. Hence, $\phi$ is surjective. \\

2. $\Longrightarrow$.  If $\lambda=0$, then for any tree $t$, there exists an integer $m_t$ such that $\phi(t)=\mu^{m_t}e_{|t|}$.
So, if $\mu \neq 0$ and $1 \in J$, taking any tree with $n$ vertices, all decorated by $1$, $e_n \in Im(\phi)$. \end{proof}\\

Using the duality between $S(\g_{CK}^J)$ dans $\h_{CK}^J$, we obtain a morphism of Hopf algebras $\phi^*$ from $S(\g_{FdB})^*$ to
$\h_{CK}^J$. More precisely:

\begin{prop}The image of $\phi^*$ is generated by the elements:
$$y(n)=\sum_{|t|=n} \frac{\mu_t}{s_t} t,\:n\geq 1.$$
If ($\lambda \neq 0$) and ($1\in J$) and ($2\in J$ or $\mu \neq \lambda$), then $Im(\phi^*)$ is isomorphic to the Faà di Bruno Hopf algebra.
\end{prop}

{\bf Remark.} The coefficients $\nu_t=\frac{\mu_t}{s_t}$ can be inductively computed in the following way:
\begin{itemize}
\item $\nu_{\tdun{$i$}}=1$.
\item If $t=B_i\left(t_1^{p_1}\ldots t_k^{p_k}\right)$, where $p_1,\ldots,p_k$ are different trees, and putting $m=p_1+\ldots+p_k$:
$$\nu_t=\frac{(\lambda j-\mu)(\lambda j)(\lambda j+\mu)\ldots (\lambda j+(m-2)\mu)}{m!}
\frac{m!}{p_1!\ldots p_k!} \nu_{t_1}^{p_1}\ldots \nu_{t_k}^{p_k}.$$
\end{itemize}
This implies that $y=\sum y(n)$ is the solution of the equation of theorem \ref{12}-1, $\lambda$ and $\mu$ being the parameters chosen 
in the definition of $\g_{FdB}$. As a consequence:

\begin{cor}\label{20}
All the Dyson-Schwinger equations of the first case in theorem \ref{12} are Hopf.
Moreover, if ($\lambda \neq 0$) and  ($1 \in J$) and ($2\in J$ or $\lambda \neq \mu$), the Hopf subalgebra associated to an equation of the first case
is isomorphic to the Faà di Bruno Hopf algebra; if ($\lambda=0$) and ($\mu \neq 0$ and $1\in J$) or ($J=\mathbb{N}^*$), it is isomorphic
to the Hopf algebra of symmetric functions $Sym$.
\end{cor}

{\bf Remark.} In the general case, let us denote by $\J$ the set of indices $i\in J$ such that $x(i) \neq 0$.
The dual Lie algebra inherits a dual basis $(e_i)_{i\in J}$, with $[e_i,e_j]=\lambda (j-i)e_{i+j}$, so it a Lie subalgebra of $\g_{FdB}$.
Dually, the Hopf subalgebra associated to an equation of the first case is isomorphic to a quotient of the Faà di Bruno Hopf algebra if $\lambda\neq 0$
or a quotient of the Hopf algebra of symmetric functions if $\lambda=0$.\\

{\bf Example.} The following example comes from \cite{Bergbauer,Kreimer2,Tanasa}:
$$x=\sum_{n\geq 1}B_n\left((1+x)^{n+1}\right),$$
where for all $n \geq 1$, $B_n$ is a $1$-cocyle of certain graded Hopf algebra, homogeneous of degree $n$.
This is a system of theorem \ref{12}, with $\lambda=1$ and $\mu=-1$, so it generates a Hopf subalgebra isomorphic to the Faà di Bruno Hopf algebra.
The isomorphism is given by:
$$\left\{\begin{array}{rcl}
\h_{FdB}&\longrightarrow& \h_{CK}^J\\
x_i&\longrightarrow&x(i).
\end{array}\right.$$

\subsection{Pre-Lie algebra associated to an equation of the second type}

\begin{prop}\label{21}
Let us consider the following equation:
$$(E):x=\sum_{\substack{j\in J\\ m\mid j}}B_j(1+\alpha x)+
\sum_{\substack{j\in J\\ m/ \hspace{-1.2mm} \mid \hspace{1.2mm} j}}B_j(1,)$$
with $\alpha \in K-\{0\}$. We put:
$$\J=\{a_1j_1+\ldots+a_kj_k\mid k\geq 1,a_1,\ldots,a_{k-1}\in \mathbb{N}^*,a_k\in \{0,1\},
j_1,\ldots, j_k \in J, m\mid j_1,\ldots, j_{k-1}\}.$$
The subalgebra generated by the components of the solution of $(E)$ is Hopf, and its dual is the enveloping algebra of a pre-Lie algebra $\g$.
The pre-Lie product of $\g$ is given in a certain basis $(f_i)_{i\in \J}$  by:
$$f_i \circ f_j=\left\{\begin{array}{l}
0 \mbox{ if } m/ \hspace{-2.6mm} \mid \hspace{2.6mm} j,\\
f_{i+j} \mbox{ if }m\mid j.
\end{array}\right.$$
It is associative. \end{prop}

\begin{proof} We already proved that these equations are Hopf, see (\ref{E2}). 
For all $n \in \mathbb{N}^*$, $x(n)$ is a linear span of ladders of degree $n$, such that the vertices are decorated by elements of $J$,
all multiples of $m$, except maybe the decoration of the leaf. It is then clear that $\J$ is the set of indices $n$ such that $x(n) \neq 0$.
As a consequence, the dual pre-Lie algebra inherits a dual basis $(e_i)_{i\in \J}$. By homogeneity,
$e_i \circ e_j=\eta_{i,j} e_{i+j}$, for a certain scalar $\eta_{i,j}$. Using the duality and (\ref{E2}):
\begin{eqnarray*}
\eta_{i,j}&=&(e_i \circ e_j)\left(x(i+j)\right)\\
&=&(e_i \otimes e_j)\left(\Delta(x(i+j))\right)\\
&=& (e_i \otimes e_j)\left(x(i+j)\otimes 1+1\otimes x(i+j)+\alpha \sum_{\substack{k=1\\m\mid k} }^{i+j-1} x(i+j-k) \otimes x(k)\right)\\
&=&\left\{\begin{array}{l}
0 \mbox{ if }m/ \hspace{-2.6mm} \mid \hspace{2.6mm} j,\\
\alpha \mbox{ if }m\mid j.
\end{array}\right.
\end{eqnarray*}

Let us take $i,j,k \in \J$.
\begin{eqnarray*}
(e_i\circ e_j)\circ e_k&=&\left\{\begin{array}{l}
\alpha^2 e_{i+j+k}\mbox{ if $m \mid j$ and $m\mid k$},\\
0\mbox{ if not};
\end{array}\right. \\
e_i\circ (e_j\circ e_k)&=&\left\{\begin{array}{l}
\alpha^2 e_{i+j+k}\mbox{ if $m \mid j$ and $m\mid j+k$},\\
0\mbox{ if not}.
\end{array}\right.
\end{eqnarray*}
So the pre-Lie product $\circ$ is associative. We now put $f_i=\frac{1}{\alpha} e_i$. The assertion on these elements is easily proved. \end{proof}\\

{\bf Remark.} $\J=\mathbb{N}^*$ if, and only if, $1,\ldots,m\in J$.

\section{Generalization of Hopf systems}

\subsection{Fundamental systems}

{\bf Notations.}\begin{enumerate}
\item For any $\beta \in K$, we put:
$$F_\beta(h)=\sum_{k=1}^\infty \frac{(1+\beta)\ldots (1+(n-1)\beta)}{n!} x^k=\left\{\begin{array}{l}
(1-\beta h)^{-\frac{1}{\beta}} \mbox{ if }\beta\neq 0,\\
e^h\mbox{ if }\beta=0.
\end{array}\right.$$
\item For all $\beta \neq -1$:
$$F_{\frac{\beta}{1+\beta}}((1+\beta)h)=\sum_{k=0}^\infty \frac{(1+\beta)\ldots (1+n\beta)}{n!} h^n,$$
so we shall put $F_{\frac{\beta}{1+\beta}}((1+\beta)h)=1$ if $\beta =-1$. 
\end{enumerate}

We first recall the definition of an extended fundamental SDSE \cite{Foissy3}:

\begin{defi}\label{22} 
An \emph{extended fundamental SDSE} has the following form: I$I$ is a set with a partition $I=I_0 \cup J_0 \cup K_0 \cup L_0\cup I_1 \cup J_1 \cup E$, such that:
\begin{itemize}
\item any part of this partition may be empty.
\item $I_0 \cup J_0$ is not empty.
\end{itemize}
We define a SDSE in the following way:
\begin{enumerate}
\item For all $i \in I_0$, there exists $\beta_i \in K$, such that:
$$f_i=F_{\beta_i}(h_i) \prod_{j\in I_0-\{i\}} F_{\frac{\beta_j}{1+\beta_j}}((1+\beta_j) h_j) \prod_{j\in J_0} F_1(h_j).$$
\item For all $i\in J_0$:
$$f_i=\prod_{j\in I_0} F_{\frac{\beta_j}{1+\beta_j}}((1+\beta_j) h_j) \prod_{j\in J_0-\{i\}} F_1(h_j).$$
\item For all $i\in K_0$:
$$f_i=\prod_{j\in I_0} F_{\frac{\beta_j}{1+\beta_j}}((1+\beta_j) h_j) \prod_{j\in J_0} F_1(h_j).$$
\item For all $i\in L_0$, there exists a family of scalars $\left(a_j^{(i)}\right)_{j\in I_0\cup J_0 \cup K_0}$, such that 
$(\exists j\in I_0, \: a^{(i)}_j \neq 1+\beta_j)$ or $(\exists j\in J_0, \: a^{(i)}_j \neq 1)$ or $(\exists j\in K_0, \: a^{(i)}_j \neq 0)$. 
$$f_i=\prod_{j\in I_0} F_{\frac{\beta_j}{a^{(i)}_j}}\left(a^{(i)}_jh_j\right)
\prod_{j\in J_0} F_{\frac{1}{a^{(i)}_j}}\left(a^{(i)}_jh_j\right) \prod_{j\in K_0} F_0\left(a^{(i)}_jh_j\right).$$
\item For all $i\in I_1$, there exists $\nu_i\in K$, a family of scalars $\left(a_j^{(i)}\right)_{j\in I_0\cup J_0 \cup K_0}$, such that 
$\nu_i \neq 1$ and, if $\nu_i \neq 0$:
$$f_i=\frac{1}{\nu_i} \prod_{j\in I_0} F_{\frac{\beta_j}{\nu_i a^{(i)}_j}}\left(\nu_i a^{(i)}_jh_j\right)
\prod_{j\in J_0} F_{\frac{1}{\nu_i a^{(i)}_j}}\left(\nu_i a^{(i)}_jh_j\right) \prod_{j\in K_0} F_0\left(\nu_i a^{(i)}_jh_j\right)+1-\frac{1}{\nu_i};$$
if $\nu_i=0$:
$$f_i=-\sum_{j\in I_0} \frac{a^{(i)}_j}{\beta_j} \ln(1-\beta_j h_j)-\sum_{j\in J_0} a^{(i)}_j \ln(1-h_j)
+\sum_{j\in K_0} a^{(i)}_j h_j+1.$$
\item For all $i \in J_1$, there exists $\nu_i\in K-\{0\}$, a family of scalars $\left(a_j^{(i)}\right)_{j\in L_0}$, 
with the following conditions:
\begin{itemize}
\item $L_0^{(i)}=\{j\in L_0\:/\:a^{(i)}_j \neq 0\}$ is not empty.
\item For all $j,k \in L_0^{(i)}$, $f_j=f_k$. In particular, we put $c^{(i)}_t=a^{(j)}_t$ for any $j\in L_0^{(i)}$, for all $t\in I_0 \cup J_0 \cup K_0$.
\end{itemize}
Then:
\begin{eqnarray*}
f_i&=&\frac{1}{\nu_i} \prod_{j\in I_0} F_{\frac{\beta_j}{c^{(i)}_j-1-\beta_j}}\left(\left(c^{(i)}_j-1-\beta_j\right)h_j\right)
\prod_{j\in J_0} F_{\frac{1}{c^{(i)}_j-1}}\left(\left(c^{(i)}_j-1\right)h_j\right)\\
&&\prod_{j\in K_0} F_0\left(c^{(i)}_jh_j\right)+\sum_{j\in L_0^{(i)}} a^{(i)}_j h_j+1-\frac{1}{\nu_i}.
\end{eqnarray*}
\item For all $i\in E$, if $a^{(i)}_j\neq 0$ and $a^{(i)}_k \neq 0$, then $f_j=f_k$; moreover:
$$f_i=1+\sum_{j\in I} a^{(i)}_j h_j.$$
The elements of $E$ are called \emph{extension vertices}.
\end{enumerate}\end{defi}

Moreover, there is a notion of \emph{level} of a vertex with the following properties:
\begin{itemize}
\item If $a^{(i)}_j\neq 0$, then the level of $i$ is $0$ or $1$ if, and only if, the level of $j$ is $0$.
\item Let $N \geq 2$. If $a^{(i)}_j \neq 0$, then the level of $i$ is $N$ if, and only if, the level of $j$ is $N-1$.
\end{itemize}
In the preceding description, the level of the vertices in $I_0\cup J_0\cup K_0 \cup L_0$ is $0$.
The level of the vertices in $I_1 \cup J_1$ is $1$. The level of an extension vertex is at least $1$.  \\

The coefficients $\lambda_n^{(i,j)}=\lambda_n^{(i,(j,1))}$ have been computed in \cite{Foissy3}. They satisfy:
$$\lambda_n^{(i,j)}=\left\{\begin{array}{l}
a^{(i)}_j \mbox{ if } n=1,\\
\tilde{a}^{(i)}_j +b_j(n-1) \mbox{ if } n>level(i),
\end{array}\right.$$
where the coefficients $b_j$, $a^{(i)}_j$ and $\tilde{a}^{(i)}_j$ are given in the following arrays:
$$b_j: \begin{array}{c|c|c|c|c|c|c|c}
j&I_0&J_0&K_0&L_0&I_1&J_1&E\\
\hline b_j&1+\beta_j&1&0&0&0&0&0
\end{array}$$
$$ a^{(i)}_j:\begin{array}{c|c|c|c|c|c|c|c}
j\setminus i&I_0&J_0&K_0&L_0&I_1&J_1&E\\
\hline I_0&1+(1-\delta_{i,j})\beta_j&1+\beta_j&1+\beta_j&a^{(i)}_j&a^{(i)}_j&(c^{(i)}_j-1-\beta_j)/\nu_i&a^{(i)}_j\\
\hline J_0&1&1-\delta_{i,j}&1&a^{(i)}_j&a^{(i)}_j&(c^{(i)}_j-1)/\nu_i&a^{(i)}_j\\
\hline K_0&0&0&0&a^{(i)}_j&a^{(i)}_j&c^{(i)}_j/\nu_i&a^{(i)}_j\\
\hline L_0&0&0&0&0&0&a^{(i)}_j&a^{(i)}_j\\
\hline I_1&0&0&0&0&0&0&a^{(i)}_j\\
\hline J_1&0&0&0&0&0&0&a^{(i)}_j\\
\hline E&0&0&0&0&0&0&a^{(i)}_j
\end{array}$$
$$ \tilde{a}^{(i)}_j:\begin{array}{c|c|c|c|c|c|c|c}
j\setminus i&I_0&J_0&K_0&L_0&I_1&J_1&E\\
\hline I_0&1+(1-\delta_{i,j})\beta_j&1+\beta_j&1+\beta_j&a^{(i)}_j&\nu_i a^{(i)}_j&c^{(i)}_j-1-\beta_j&\tilde{a}^{(i)}_j\\
\hline J_0&1&1-\delta_{i,j}&1&a^{(i)}_j&\nu_i a^{(i)}_j&c^{(i)}_j-1&\tilde{a}^{(i)}_j\\
\hline K_0&0&0&0&a^{(i)}_j&\nu_i a^{(i)}_j&c^{(i)}_j&\tilde{a}^{(i)}_j\\
\hline L_0&0&0&0&0&0&0&0\\
\hline I_1&0&0&0&0&0&0&0\\
\hline J_1&0&0&0&0&0&0&0\\
\hline E&0&0&0&0&0&0&0
\end{array}$$

For the extension vertices, the way to compute $\tilde{a}^{(i)}_j$ is the following.
Let us choose a $i'$ such that $a^{(i)}_{i'} \neq 0$. For all $n \geq 1$, $\lambda^{(i,j)}_{n+1}=\lambda^{(i',j)}_n$. In particular, if $n> level(i)$:
$$\tilde{a}^{(i)}_j+b_jn=\tilde{a}^{(i')}_j+b_j(n-1).$$
So $\tilde{a}^{(i)}_j=\tilde{a}^{(i')}_j-b_j$. With an induction on the level, this implies that $\tilde{a}^{(i)}_j=0$
for any $i\in J$ if $j\in L_0 \cup I_1 \cup J_1\cup E$.

\begin{theo}\label{23}
 Let $(S)$ a Hopf SDSE such that the truncation at $1$ gives a fundamental extended system. Let us put:
$$Q(h)=\prod_{j\in I_0} (1-\beta_j h_j)^{-\frac{1+\beta_j}{\beta_j}}\prod_{j\in J_0} (1-h_j)^{-1}.$$
Then for all $i\in I$, there exists a formal series $g^{(i)}$, depending only of the $h_j$'s with $j \in I_0 \cup J_0 \cup K_0$,
such that if $q > level(i)$, then $f^{(i,q)}=g^{(i)} Q^q$. In particular:
\begin{enumerate}
\item If $i\in I_0$, $g^{(i)}=1-\beta_i h_i$.
\item If $i\in J_0$, $g^{(i)}=1-h_i$.
\item If $i\in K_0$, $g^{(i)}=1$.
\item If $i\in E$, $g^{(i)}=Q(h)^{-1}$.
\end{enumerate}\end{theo}

\begin{proof} We apply lemma \ref{7}. As $q>level(i)$, this gives:
$$a^{(i,q)}_{(p_1,\ldots,p_j+1,\ldots,p_N)}=\left(\tilde{a}^{(i)}_j+b_j(p_1+\ldots+p_N+q-1)
-\sum_{l=1}^N a^{(l,1)}_j p_l\right) a^{(i,q)}_{(p_1,\ldots,p_N)}.$$
In particular, for $p_1=\ldots=p_N=0$,  $a^{(i,q)}_j=\tilde{a}^{(i)}_j+b_j(q-1)$. 
If $j \in L_0\cup I_1 \cup J_1 \cup E$, then $\tilde{a}^{(i)}_j=b_j=0$, so $a^{(i,q)}_j=0$ and $f^{(i,q)}$ does not depend
on $h_j$. From now, we assume that $p_k=0$ if $k\in L_0 \cup I_1 \cup J_1 \cup E$.
Let us take $j\in I_0 \cup J_0 \cup K_0$. For all $l\in I_0 \cup J_0 \cup K_0$, $a^{(l)}_j=b_j$, except perhaps if $l=j$. So:
\begin{eqnarray*}
a^{(i,q)}_{(p_1,\ldots,p_j+1,\ldots,p_N)}&=&\left(\tilde{a}^{(i)}_j+b_j(p_1+\ldots+p_N+q-1)
-\sum_{l=1}^N b_j p_l+\left(b_j-a^{(j)}_j\right)p_j\right) a^{(i,q)}_{(p_1,\ldots,p_N)}\\
&=&\left(\tilde{a}^{(i)}_j+b_j(q-1)+\left(b_j-a^{(j)}_j\right)p_j\right) a^{(i,q)}_{(p_1,\ldots,p_N)}.
\end{eqnarray*}
This implies:
\begin{eqnarray}
\nonumber f^{(i,q)}&=&\prod_{j\in I_0\cup J_0 \cup K_0}\left(1-\left(b_j-a^{(j)}_j\right)h_j\right)^
{-\frac{\tilde{a}^{(i)}_j-b_j}{b_j-a^{(j)}_j}-\frac{b_j}{b_j-a^{(j)}_j}q}\\
\label{E3} &=&\underbrace{\prod_{j\in I_0}(1-\beta_j h_j)^{-\frac{\tilde{a}^{(i)}_j-1-\beta_j}{\beta_j}}
\prod_{j\in J_0}(1-h_j)^{-\tilde{a}^{(i)}_j+1} \prod_{j\in K_0}e^{-\tilde{a}^{(i)}_j h_j}}_{g^{(i)}}\\
\nonumber &&\left(\underbrace{\prod_{j\in I_0} (1-\beta_j h_j)^{-\frac{1+\beta_j}{\beta_j}}\prod_{j\in J_0} (1-h_j)^{-1}}
_{Q(h)} \right)^q.
\end{eqnarray}
In particular, if $i,j\in I_0 \cup J_0 \cup K_0$, $\tilde{a}^{(i)}_j=a^{(i)}_j=b_j$, except perhaps if $i=j$. In this case:
$$g^{(i)}=\left((1-\left(b_i-a^{(i)}_i\right)h_i\right)^{-\frac{\tilde{a}^{(i)}_i-b_i}{b_i-a^{(i)}_i}}
=\left\{\begin{array}{l}
(1-\beta_i h_i)\mbox{ if }i\in I_0,\\
(1-h_i) \mbox{ if }i\in J_0,\\
1\mbox{ if }i\in K_0.
\end{array}\right.$$
If $i\in E$, then $\tilde{a}^{(i)}_j=0$ for all $j\in I$, so:
$$g^{(i)}=\prod_{j\in I_0} (1-\beta_j h_j)^{\frac{1+\beta_j}{\beta_j}}\prod_{j\in J_0} (1-h_j).$$
Hence, $g^{(i)}=Q(h)^{-1}$. \end{proof}\\

In the particular case where $L_0=I_1=J_1=E=\emptyset$:

\begin{cor}\label{24}
Let us take $I=I_0 \cup J_0 \cup K_0$, with $I_0 \cup J_0 \neq \emptyset$.
The following SDSE is Hopf: 
\begin{itemize}
\item For all $i\in I_0$:
$$x_i=\sum_{q\in J_i} B_{(i,q)}\left((1-\beta_i x_i) \prod_{j\in I_0} (1-\beta_jx_j)^{-\frac{1+\beta_j}{\beta_j}q}
\prod_{j\in J_0}(1-x_j)^{-q}\right).$$
\item For all $i\in J_0$:
$$x_i=\sum_{q\in J_i} B_{(i,q)}\left((1-x_i) \prod_{j\in I_0} (1-\beta_jx_j)^{-\frac{1+\beta_j}{\beta_j}q}
\prod_{j\in J_0}(1-x_j)^{-q}\right).$$
\item For all $i\in K_0$:
$$x_i=\sum_{q\in J_i} B_{(i,q)}\left( \prod_{j\in I_0} (1-\beta_jx_j)^{-\frac{1+\beta_j}{\beta_j}q}
\prod_{j\in J_0}(1-x_j)^{-q}\right).$$
\end{itemize}\end{cor}

{\bf Example.} For the example of the introduction:
$$\left\{\begin{array}{rcl}
x_1&=&\displaystyle \sum_{k\geq 1}B_{(1,k)}\left(\frac{(1+x_1)^{1+2k}}{(1-x_2)^k(1-x_3)^{2k}}\right),\\[3mm]
x_2&=&\displaystyle B_2\left(\frac{(1+x_1)^2}{(1-x_3)^2}\right),\\[3mm]
x_3&=&\displaystyle B_3\left(\frac{(1+x_1)^2}{(1-x_2)(1-x_3)}\right).
\end{array}\right.$$
This is obtained from a fundamental system, with $I_0=\{1,3\}$, $J_0=\{2\}$, $\beta_1=-1/3$, $\beta_3=1$,
by a change of variables $h_1\longrightarrow 3h_1$. Hence, it generates a Hopf subalgebra. \\

Corollary \ref{24} determines the formal series $f^{(i,q)}$ when $i \in I_0\cup J_0 \cup K_0$. If $i \in L_0$, 
theorem \ref{23} and (\ref{E3}) determines all the $f^{(i,q)}$.
If $i\in I_1 \cup J_1$, then theorem \ref{23} and (\ref{E3}) determines the $f^{(i,q)}$ if $q\geq 2$,
and $f^{(i,1)}$ is given in definition \ref{22}. It remains to determine $f^{(i,q)}$ when $i\in E$.\\

{\bf Notations.}\begin{enumerate}
\item  Let $(S)$ be a Hopf SDSE, and let $\gs$ be the (oriented) \emph{graph of dependence }of $(S)$, that is to say:
\begin{itemize}
\item The vertices of $\gs$ are the elements of $I$.
\item There is an oriented edge from $i$ to $j$ if, and only if, $a^{(i,1)}_j \neq 0$.
\end{itemize}
\item Let $i,j \in J$. We shall write $i\fleche{}j$ if there is an oriented edge from $i$ to $j$ in $\gs$.
For all $q\geq 1$, we shall write $i\fleche{q} j$ if there is an oriented path of length $q$ from $i$ to $j$ in $\gs$.
In particular, $i\fleche{0} j$ if, and only if, $i=j$.
\end{enumerate}

Theorem \ref{23} and its proof give all the formal series $f^{(i,q)}$ when the level of $i$ is $\leq 1$. 
If $i$ is an extension vertex, its level can be greater than $2$, and the associated formal series are now described:

\begin{prop}\label{25}
Let $(S)$ be a Hopf SDSE such that the truncation at $1$ is an extended fundamental system. For any $i\in E$, of level $n\geq 1$:
\begin{itemize}
\item If $q < n$, $\displaystyle f^{(i,q)}=1+\sum_{j\in I} a^{(i',1)}_j h_j$, where $i'$ is any element of $I$ such that $i\fleche{q-1}i'$.
\item $f^{(i,n)}=f^{(i',1)}$, where $i'$ is any element of $J$ such that $i\fleche{n-1}i'$.
\item If $q>n$, then $f^{(i,q)}=Q^{q-1}$. 
\end{itemize}\end{prop}

\begin{proof} If $q>n$, this comes directly from theorem \ref{23}. Let us prove the case $q<n$ by induction on $n$.
If $n=1$, there is nothing to prove. Let us assume the results at all rank $k<n$. Let $i_0=i \rightarrow i_1 \rightarrow \cdots \rightarrow i_n$
in the graph of dependence $\gs$. As the level of $i_0$ is $n$, the level of $i_k$ is $n-k$ for all $0 \leq k \leq n$.
In particular, $i_1,\ldots,i_{n-2}$ are extension vertices, as their level is $\geq 2$. We apply proposition \ref{6}  with 
$t=B_{(i_0,1)}\circ \ldots \circ B_{(i_k,1)}(1)$, with $k\leq n-1$. As $i_0,\ldots,i_{k-1}$ are extension vertices,
the only tree $t'$ such that $a_{t'}\neq 0$ and $n_{(j,1)}(t,t')\neq 0$ is $B_{(i_0,1)}\circ \ldots \circ B_{(i_k,1)}\circ B_{(j,1)}(1)$.
So:
$$\lambda_{k+1}^{(i,(j,1))}a^{(i_0,1)}_{i_1}\ldots a^{(i_{k-1},1)}_{i_k}
=a^{(i_0,1)}_{i_1}\ldots a^{(i_{k-1},1)}_{i_k}a^{(i_k,1)}_j.$$
Hence, if $1\leq l \leq n$, $\lambda^{(i,(j,1))}_l=a^{(i',1)}_j$, where $i'$ is any element of $I$ such that $i\fleche{l-1}i'$. 
As a consequence, if $\lambda^{(i,(j,1))}_l \neq 0$, $i\fleche{l} j$, so the level of $j$ is $n-l$.\\

Let us fix $1\leq q\leq n$. From lemma \ref{7}, with $p_1=\ldots=p_N=0$, for all $j\in J$, $a^{(i,q)}_j=\lambda_q^{(i,(j,1))}$.
So this is equal to $a^{(i')}_j$ for any $i'$ such that $i\fleche{q-1}i'$. Hence, if $a^{(i,q)}_j \neq 0$, then  $i\fleche{q} j$.
So the level of $j$ is $n-q$.

If $q<n$, let us now consider $j,j'\in J$. If $a^{(i,q)}_j=0$ or $a^{(i,q)}_{j'}=0$, by lemma \ref{7}, $a^{(i,q)}=0$.
If $a^{(i,q)}_j\neq 0$ and $a^{(i,q)}_{j'}\neq 0$, then 
$a^{(i,q)}_{j,j'}=\left(\lambda^{(i,(j',1))}_{q+1}-a^{(j,1)}_{j'}\right) a^{(i,q)}_{j'}$.
As the level of $j$ and $j'$ is $n-q\geq 1$, $a^{(j,1)}_{j'}=0$. Moreover, as the level of $j'$ is $n-q\neq n-q-1$,
$\lambda^{(i,(j',1))}_{q+1}=0$. So $a^{(i,q)}_{j,j'}=0$. So all the terms of degree $2$ of $f^{(i,q)}$ are zero,
so all the terms of degree $\geq 2$ of $f^{(i,q)}$ are zero. \\

Let us finish with the case $q=n$. Let us choose $i'$ such that $i \fleche{n-1} i'$; then the level of $i'$ is $1$.
We already saw that $a^{(i,n)}_j=a^{(i',1)}_j$ for all $j\in J$.
If $i=i_0\rightarrow \cdots \rightarrow i_{n-1}=i'$ in $\gs$, then $i_0,\ldots,i_{n-2}$ are extension vertices as their level is $\geq 2$, so for all $j\in J$:
$$\tilde{a}^{(i)}_j=\tilde{a}^{(i_1)}_j-\beta_j=\ldots=\tilde{a}^{(i')}_j-(n-1)\beta_j.$$
 Let us apply lemma \ref{7}, with $(p_1,\ldots,p_N)\neq (0,\ldots,0)$. Then $p_1+\ldots+p_n+n>n$, so:
\begin{eqnarray*}
&&a^{(i,n)}_{(p_1,\ldots,p_j+1,\ldots,p_N)}\\
&=&\frac{1}{p_j+1}\left(\tilde{a}^{(i')}_j-(n-1)b_j+
b_j(p_1+\ldots+p_N+n-1)-\sum_{l=1}^N a^{(l,1)}_j p_l\right)a^{(i,n)}_{(p_1,\ldots,p_N)}\\
&=&\frac{1}{p_j+1}\left(\tilde{a}^{(i')}_j+b_j(p_1+\ldots+p_N)
-\sum_{l=1}^N a^{(l,1)}_j p_l\right)a^{(i,n)}_{(p_1,\ldots,p_N)}\\
&=&\frac{1}{p_j+1}\left(\lambda^{(i',(j,1))}_{p_1+\ldots+p_N+1}
-\sum_{l=1}^N a^{(l,1)}_j p_l\right)a^{(i,n)}_{(p_1,\ldots,p_N)}.
\end{eqnarray*}
By lemma \ref{7}, this is the same induction as the coefficients $a^{(i',1)}_{(p_1,\ldots,p_N)}$. So $f^{(i,n)}=f^{(i',1)}$.  \end{proof}\\

{\bf Remark.} If $q<n$, the level of $i'$ is $n-q+1\geq 2$, so $i'$ is an extension vertex, and $f^{(i,q)}=f^{(i',1)}$.

\subsection{Quasi-cyclic systems}

Let us first recall the structure of a quasi-cyclic SDSE:

\begin{defi} 
Let $N \geq 2$. A SDSE is \emph{$N$-quasi-cyclic} if it has the following form:
$J$ admits a partition $J=J_{\overline{1}}\cup\cdots \cup J_{\overline{N}}$ indexed by $\mathbb{Z}/N\mathbb{Z}$, with the following conditions:
\begin{enumerate}
\item If $i \in J_{\overline{p}}$, its direct descendants are all in $J_{\overline{p+1}}$.
\item If $i$ and $j$ have a common direct ascendant, then they have the same direct descendants.
\end{enumerate}
Moreover, for all $i\in J$:
$$f_i=1+\sum_{i\longrightarrow j} a^{(i)}_j h_j,$$
and if $i$ and $j$ have a common direct ascendant, then $f_i=f_j$. 
\end{defi}

{\bf Remark.} Let us consider a quasi-cyclic SDSE.
For all $i \in J$:
$$x_i=\tdun{$i$}+\sum_{n=1}^\infty \sum_{i_1,\ldots,i_n \in I} a^{(i)}_{i_1}a^{(i_1)}_{i_2}\ldots a^{(i_{n-1})}_{i_n}
B_i\circ B_{i_1}\circ \ldots \circ B_{i_n}(1).$$

In order to simplify the problem, we shall assume that for all $i,j \in J$ such that $i\rightarrow j$ in $\hs$, then $a^{(i)}_j$ depends only on $i$.
This is generally not the case, but it can be assumed without loss of generality if there is no vertex with no ascendant in the graph of dependence of the system.
In this case, if $i=i_0 \rightarrow \ldots \rightarrow i_n$ is a path of length $n$ in $\gs$,  
$a^{(i_0)}_{i_1}\ldots a^{(i_{n-1})}_{i_n}$ only depends on $i$ and $n$: we denote it by $b^{(i)}_n$. Then for all $i\in I$:
$$x_i=\sum_{n=0}^{\infty} \sum_{i\rightarrow i_1 \rightarrow \cdots \rightarrow i_n}
b^{(i)}_n B_i \circ B_{i_1}\circ \ldots \circ B_{i_n}(1),$$
with the convention $b^{(i)}_0=1$.  Moreover, if there is a path of length $m$ from $i$ to $j$, then 
$b^{(i)}_m b^{(j)}_n=b^{(i)}_{m+n}$ for all $n$.
 With these notations, $\lambda_n^{(i,j)}=0$ if there is no path of length $n$ from $i$ to $j$ and is equal to $b^{(i)}_n/b^{(i)}_{n-1}$
if there is such a path. 

\begin{theo}
Let $(S)$ be a Hopf SDSE such that the truncation at $1$ gives a quasi-cyclic SDSE satisfying the preceding hypothesis.
Then, for all $i\in J$, for all $q \in I_i$:
$$f^{(i,q)}=1+b_q^{(i)}\sum_{i\fleche{q} j}h_j.$$
\end{theo}

\begin{proof} Up to a change of indexation, we shall assume that $i\in I_{\overline{0}}$. From lemma \ref{7},
for all $j\in J$, $a^{(i,q)}_j=\lambda^{(i,(j,1))}_q$. So this is $0$ if there is no path from $i$ to $j$ of length $q$ and equal to $b^{(i)}_q$
if $i\fleche{q} j$.
Let $j,k \in J$. If $a^{(i,q)}_j=0$ or $a^{(i,q)}_k=0$, then $a^{(i,q)}_{j,k}=0$.
Let us assume that $a^{(i,q)}_j, a^{(i,q)}_k \neq 0$. Then $i\fleche{q}j,k$, so $j,k \in I_{\overline{q}}$. From lemma \ref{7}:
$$a^{(i,q)}_{j,k} =\left(\lambda_{q+1}^{(i,(j,1))}-a^{(k)}_j \right)a^{(i,q)}_k.$$
As $j\in I_{\overline{q}}$, $j\notin I_{\overline{q+1}}$ so there is no path of length $q+1$ from $i$ to $j$, and $\lambda_{q+1}^{(i,(j,1))}=0$.
As $j,k$ are both in $I_{\overline{q}}$, $a^{(k)}_j=0$. As a consequence, $a^{(i,q)}_{j,k}=0$.
All the terms of degree $2$ of $f^{(i,q)}$ are equal to $0$, so all the terms of degree $\geq 2$ of $f^{(i,q)}$ are equal to $0$. \end{proof}\\

{\bf Remark.} If this holds, for all $i$, $n \in \mathbb{N}^*$, $x_i(n)/b^{(i)}_n$ is a sum of the ladders 
$B_{(i_1,p_1)}\circ \ldots \circ B_{(i_k,p_k)}(1)$ with the following conditions:
\begin{itemize}
\item $i_1=1$.
\item $p_1+\ldots+p_k=n$.
\item for all $1\leq r\leq k-1$, there exists a path of length $p_r$ from $i_r$ to $i_{r+1}$ in $\gs$.
\end{itemize}

{\bf Example.} Here is an example of such a SDSE. For all $\overline{i}\in \mathbb{Z}/M\mathbb{Z}$,
let us choose $J_{\overline{i}} \subseteq \mathbb{N}^*$. Then:
$$x_{\overline{i}}=\sum_{j\in J_{\overline{i}}} B_j(1+x_{\overline{i+j}}).$$

\bibliographystyle{amsplain}
\bibliography{biblio}

\end{document}